\documentclass[11pt,a4paper]{article}

\usepackage{amsmath}
\usepackage{amsthm}

\usepackage{newtxtext}
\usepackage{newtxmath}
\usepackage{bm}

\usepackage[a4paper,margin=1in]{geometry}
\usepackage{graphicx}
\usepackage{float}
\usepackage{booktabs}
\usepackage{enumitem}
\usepackage{caption}
\usepackage{algorithm}
\usepackage{algpseudocode}

\usepackage{tikz}
\usetikzlibrary{positioning,arrows.meta,shapes.geometric,trees,fit,backgrounds}

\theoremstyle{plain}
\newtheorem{theorem}{Theorem}
\theoremstyle{definition}
\newtheorem{remark}{Remark}
\makeatletter
\renewenvironment{proof}[1]{%
  \par\medskip\noindent\textit{#1}\enspace\ignorespaces
}{%
  \par\medskip
}
\makeatother
\newcommand{\Halmos}{\hspace*{\fill}\ensuremath{\square}}

\usepackage[authoryear,round]{natbib}
\bibpunct[, ]{(}{)}{,}{a}{}{,}%

\newcommand{\FIGURE}[3]{%
  \centering
  #1\par
  \caption{#2}%
  \par\smallskip{\footnotesize #3\par}%
}
\newcommand{\ACKNOWLEDGMENT}[1]{\par\bigskip\noindent\textbf{Acknowledgment.}\quad #1\par}
\newcommand{\EMAIL}[1]{\texttt{#1}}

\usepackage[colorlinks=true,linkcolor=blue!55!black,citecolor=blue!55!black,urlcolor=blue!55!black]{hyperref}
\hypersetup{%
  pdftitle={Simulation-Optimization of Systems of Optimizers: Exploiting the Inner Optimization's Geometry},
  pdfauthor={Zhou He}
}

\title{Simulation-Optimization of Systems of Optimizers:\\Exploiting the Inner Optimization's Geometry}
\author{Zhou He\\[4pt]\small University of Chinese Academy of Sciences.\quad\EMAIL{hezhou@ucas.ac.cn}}
\date{\today}

\begin{document}
\maketitle

\begin{abstract}
We study \emph{simulation-optimization of systems of optimizers} (SOSO): agent-based simulations in which every agent solves a structured optimization---a linear program (LP), mixed-integer program, or dynamic program---at every decision epoch. Such systems are ubiquitous in multi-echelon supply chains, wholesale electricity markets, and automated logistics, yet standard simulation optimization treats the simulation as a black box and discards the inner optimization's geometry. We formalize SOSO and develop a framework that converts this geometry into computational advantage. We prove that the inner LP's optimal basis and dual variables propagate upward through the stochastic dynamics to yield an exact, unbiased, single-replication infinitesimal perturbation analysis (IPA) gradient of the outer objective, with detectable, measure-zero basis changes as the only obstructions. We derive a common-random-numbers covariance bound governed by a computable basis-disagreement count that is a free by-product of the forward simulation. We further prove that IPA variance grows exponentially with feedback depth---a formalization of the bullwhip effect in gradient estimation---and introduce surrogate-as-decision with an error budget that unifies LP-horizon and reinforcement-learning surrogates under a single bound. We compose these results into PRIME, a stochastic-approximation solver integrating the IPA gradient, adaptive step sizing, and multi-start spatial diversification. On six testbeds spanning the SOSO taxonomy, PRIME achieves the best or tied-best optimality gap at equal budget with near-zero oscillation and narrow seed-to-seed variance; the IPA gradient provides a ${\sim}2000{\times}$ per-replication variance reduction over independent finite differences at a representative policy point; and common random numbers cut paired-difference variance by more than $10{,}000{\times}$ (peak, near capacity saturation) in a 1{,}000-SKU, six-distribution-center supply chain. Black-box gradient methods suffer exponential variance growth under feedback amplification, while PRIME remains stable. These results establish exploiting embedded-optimization geometry as a practically significant direction for simulation optimization.
\end{abstract}

\noindent\textbf{Keywords:} Simulation optimization, systems of optimizers, infinitesimal perturbation analysis, agent-based simulation, common random numbers

\section{Introduction}\label{sec:intro}

\subsection{The SOSO problem: definition and decision process}

Consider a planner who controls a shared, low-dimensional policy parameter $\theta\in\Theta$---a price, a tax, a capacity threshold, or a service-level target---and a system of $N$ agents that operates over $T$ decision epochs. At each epoch, every agent $i$ observes its local state $s_{i,t}$ and chooses its action $a_{i,t}$ by solving a \emph{structured optimization problem} whose parameters depend on $s_{i,t}$ and on $\theta$: a linear program (LP), a mixed-integer linear program (MILP), or a dynamic program (DP) or reinforcement-learning (RL) policy. The system then evolves stochastically and path-dependently, and the planner seeks the policy $\theta^{*}\in\arg\min_{\theta\in\Theta}J(\theta)$ that minimizes the expected total cost $J(\theta)=\mathbb{E}_\xi[L(\theta,\xi)]$. Because the dynamics are stochastic and admit no closed form, $J(\theta)$ can be evaluated only by simulating the system, and the planner must optimize over $\theta$ from simulation output. We call this the \emph{Simulation-Optimization of Systems of Optimizers} (SOSO) problem. The name is deliberate: each agent is itself an optimizer, so optimization is nested at three levels---inside each agent at each epoch, inside the simulation that threads these solves together, and outside it over the policy $\theta$ (Figure~\ref{fig:soso-nesting}). The per-period action is the solution of a parametric optimization, not a fixed reaction function, and it is this feature---not the event logic of a conventional discrete-event simulation---that drives both the computational cost and the mathematical structure of the problem.

Three concrete decision contexts instantiate this structure and motivate the inner-problem types we study:
\begin{enumerate}
\item \emph{Multi-echelon supply chains.} Each firm or distribution center solves a production- or transfer-planning LP against forecast demand \citep{simchilevi2014}, with fixed resource-consumption rates and state-dependent right-hand sides such as on-hand inventory and outstanding orders. The outer parameter $\theta$ is a transfer-capacity limit, a holding-cost weight, or an order-up-to threshold; the planner's objective is total inventory, stockout, and transportation cost.

\item \emph{Wholesale electricity markets.} Each generator solves a unit-commitment MILP with binary start-up variables and continuous output levels, given posted prices and network constraints \citep{hobbs2001}. The outer parameter $\theta$ is a carbon tax, a price cap, or a renewable subsidy; the planner's objective is social cost or reliability.

\item \emph{Automated logistics and learned policies.} Dispatchers solve minimum-cost flow LPs or vehicle-routing MILPs, while an increasing number of agents execute learned DP/RL policies for routing, pricing, or admission control \citep{sutton2018,powell2022}. The outer parameter $\theta$ controls fleet size, service-level targets, or commission rates.
\end{enumerate}
Of the three inner-problem types, this paper develops the full gradient and CRN machinery for LP inners (Sections~\ref{sec:problem}--\ref{sec:theory}), extends the surrogate-as-decision error budget to DP/RL inners (Section~\ref{sec:eps}); for MILP inners, the error-budget framework extends in principle through LP-relaxation surrogates, but exact IPA for MILP inners remains an open problem (Section~\ref{sec:limitations}). The MILP cells of the taxonomy (Figure~\ref{fig:soso-taxonomy}) are therefore left to future work. In every case the per-replication cost of the simulation is dominated by the inner optimization solves---$N$ agents $\times$ $T$ periods $\times$ $M$ replications $\times$ $K$ candidate policies, routinely $10^{7}$--$10^{9}$ inner solves at industrial scale---so the geometry of those inner problems, not the event logic of the simulation, is the dominant computational object.

\begin{figure}[t]
\centering
\begin{tikzpicture}[font=\small, >=Stealth,
  every node/.style={align=center},
  outerbox/.style={draw=blue!55!black, line width=0.9pt, fill=blue!5},
  simbox/.style={draw=green!45!black, line width=0.8pt, fill=green!7},
  innbox/.style={draw=orange!60!black, line width=0.8pt, fill=orange!10},
  tab/.style={font=\small\bfseries, fill=white, inner sep=2pt},
]
\draw[outerbox] (-5.0,-2.0) rectangle (5.0,3.0);
\node[tab, text=blue!55!black] at (0,3.0) {Outer level --- simulation optimization};
\node[font=\footnotesize] at (0,2.2) {Planner:\ $\theta^{*}\in\arg\min_{\theta\in\Theta}J(\theta)$,\quad $J(\theta)=\mathbb{E}_\xi[L(\theta,\xi)]$};

\draw[simbox] (-3.8,-1.7) rectangle (3.8,1.5);
\node[tab, text=green!40!black] at (0,1.5) {Simulation level --- one replication $\xi$};
\node[font=\footnotesize] at (0,0.9) {$t=0,\dots,T{-}1$;\ \ agents $i=1,\dots,N$};
\node[font=\footnotesize] at (0,0.4) {$s_{i,t+1}=\phi_i(s_{i,t},a^{*}_{i,t},\xi_t,\theta)$,\quad reward $\pi_i$};

\draw[innbox] (-3.4,-1.5) rectangle (3.4,-0.2);
\node[tab, text=orange!55!black] at (0,-0.2) {Inner level --- per agent, per epoch};
\node[font=\footnotesize] at (0,-0.75) {$a^{*}_{i,t}=\arg\max_a\,c(s_{i,t},\theta)^{\top}a$\quad\text{s.t.}\ $Ga\le h(s_{i,t},\theta)$};
\node[font=\scriptsize, text=orange!55!black] at (0,-1.18) {solver returns $a^{*}$, dual $\lambda$, basis $B_k$};

\draw[->, line width=0.9pt, blue!55!black] (-4.3,1.95) -- (-4.3,-1.95);
\node[font=\footnotesize, text=blue!55!black, rotate=90, anchor=south] at (-4.35,0) {policy $\theta$};

\draw[->, line width=0.9pt, gray!45!black] (4.3,-1.95) -- (4.3,1.95);
\node[font=\footnotesize, text=gray!30!black, rotate=-90, anchor=south] at (4.35,0) {cost $L$, gradient $\widehat{J}'_{\mathrm{IPA}}$};
\end{tikzpicture}
\caption{The SOSO decision process, with optimization nested at three levels. The \emph{outer level} is a simulation-optimization loop minimizing $J(\theta)=\mathbb{E}_\xi[L(\theta,\xi)]$. Each evaluation runs a \emph{simulation}---one replication $\xi$ over $T$ periods and $N$ agents---whose state evolves stochastically and path-dependently. At every epoch each agent solves an \emph{inner} structured optimization (LP, MILP, or DP/RL) parameterized by its state $s_{i,t}$ and by $\theta$; the solver returns the action $a^{*}_{i,t}$ together with the dual $\lambda$ and optimal basis $B_k$ that this paper exploits.}\label{fig:soso-nesting}
\end{figure}
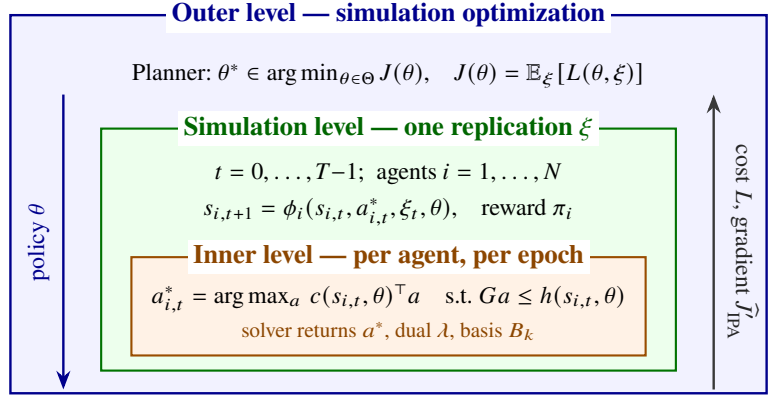

\subsection{Why SOSO is not a single centralized optimization}\label{sec:centralized}

A natural objection to the formulation above is to ask why the planner does not simply collect every agent's optimization into one large problem and solve jointly for all actions, all periods, and all scenarios at once---a centralized multistage stochastic program in which a single solver with full information chooses the entire trajectory $\{a_{i,t}\}_{i,t}$. This formulation is mathematically clean, and for small systems it is the right thing to do. For the systems that motivate SOSO it is often impractical, for four reinforcing reasons.

\paragraph{Computational intractability at scale.}
Aggregating $N$ agents over $T$ periods yields a joint decision vector of dimension $N\cdot T\cdot n$ (with $n$ the per-agent action dimension) inside a scenario tree whose node count grows exponentially in $T$ under stochastic demand \citep{shapiro2014}. Even when every agent's subproblem is a small LP, the aggregated problem is a large-scale stochastic program whose solve time grows superlinearly in its dimension; introduce the binary variables of unit commitment or vehicle routing and the aggregated problem becomes a MILP whose worst-case complexity is exponential. The systems SOSO targets---a thousand-SKU distribution network over a month, a regional electricity market over a year---are precisely those for which the aggregated problem is often too large to solve within the planning horizon it is meant to govern. Simulation, by contrast, decomposes the joint problem back into $N\cdot T$ small inner solves and pays only for the replications it needs.

\paragraph{Information that is local, real-time, and not aggregable.}
Each agent holds state---on-hand inventory, machine availability, a just-arrived demand signal, a local price quote---that the planner neither observes directly nor can collect in time to act on. A centralized solver requires all of this information to flow up to the center before each decision, which imposes communication cost, latency, and staleness, and in strategic settings (competing firms, regulated generators) the agents may be unwilling to reveal their private state at all. The SOSO formulation respects this asymmetry directly: the planner broadcasts only the low-dimensional policy $\theta$, and each agent combines $\theta$ with its own local state inside its own optimization. Global information (the policy) travels one way; local information stays where it originates.

\paragraph{Autonomy and incentive.}
In the deployed settings above, the agents are not subroutines of a single objective but autonomous decision makers---independent firms in a supply chain, separate generating companies in a market, separately managed distribution centers. A planner who chooses every action directly would require agents to disclose private information and would centralize decisions that are currently made using local expertise, potentially reintroducing principal--agent frictions. Setting a policy $\theta$ and letting each agent optimize against it is the standard practical resolution: the planner shapes the rules of the game---a transfer price, a capacity allowance, a service-level target---and the agents, who best understand their local tradeoffs, execute within them.

\paragraph{The equilibrium alternative is also limited.}
A second natural reaction is to model the agents as players in a game and compute the equilibrium: a Nash--Cournot equilibrium in an electricity market \citep{hobbs2001,metzler2003nash,kolstad1991computing} or, more generally, a complementarity solution \citep{facchinei2003finite,gabriel2013}. Equilibrium models are attractive when each agent's reaction is a closed-form, state-independent function and the system settles to a static fixed point. The systems we study do neither: demand is stochastic and seasonal, inventories and prices are state variables that carry history forward, and the planner's objective depends on the full trajectory of costs rather than on a single-period equilibrium. With path-dependent, non-stationary dynamics no closed-form equilibrium exists, and the system's response to $\theta$ can be evaluated only by simulating the trajectory. This rules out a complementarity reformulation and is, ultimately, what makes the problem one of simulation optimization rather than of mathematical programming.

\smallskip
SOSO thus gives the planner a single lever---the policy $\theta$---through which to balance centralized control against local responsiveness: the planner sets $\theta$ and optimizes it in the outer loop, while each agent retains full authority over its own inner optimization and reacts rapidly to its local state. The technical question the rest of the paper addresses is how to make that outer optimization tractable---how to exploit, rather than discard, the structure of the inner optimizations the agents are already solving.

\subsection{Related work}\label{sec:gaps}

Adopting this formulation, the planner must place the simulation inside a simulation-optimization outer loop \citep{hong2009,hong2015discrete,andradottir1998}. Every replication of $L(\theta,\xi)$ is itself a sequence of inner optimizations, so the total cost can reach $N\cdot T\cdot M\cdot K$ inner solves. We review six streams of related work and identify the gaps that motivate our approach.

\paragraph{IPA and gradient estimation.}
Infinitesimal perturbation analysis (IPA) was developed for queueing and manufacturing systems by \citet{glasserman1991,hocao1991,lecuyer1990}. In these classical settings, the gradient source is the simulation's event logic (arrival times, service completions). \citet{glasserman1995sensitivity} applied IPA to base-stock levels in multi-echelon production-inventory systems---the closest predecessor to our work. However, their base-stock policy is a fixed decision rule, not an embedded LP whose basis carries gradient data. \citet{fu1994sample} developed sample-path derivatives for $(s,S)$ inventory systems, again treating the policy as a rule rather than an optimization. No existing IPA variant reads the gradient from an inner optimization's parameters.

\paragraph{Simulation optimization.}
Discrete simulation optimization (SO) via ranking and selection, stochastic approximation, and metamodel methods is surveyed by \citet{hong2015discrete,andradottir1998}. Variance reduction techniques---adaptive control variates \citep{kim2007adaptive}, common random numbers (CRN) \citep{glasserman1992some,lecuyer1994efficiency}, and knowledge-gradient methods \citep{frazier2009knowledge}---improve estimation efficiency but do not exploit inner-optimization structure. Stochastic Kriging \citep{ankenman2010,kleijnen2008} surrogates the outer objective $J(\theta)$, which is orthogonal to our surrogate-as-decision (surrogating the agent's decision $a^{*}(s)$). Most treat the simulation as a black box in the sense that they do not exploit the inner optimization's basis or dual information.

\paragraph{Parametric LP sensitivity.}
The theory of parametric linear programming partitions the $(s,\theta)$-space into critical regions on each of which the optimal basis is constant \citep{gal1995,bazaraa2010}. LP postoptimality analysis \citep{freund1985postoptimal} and degeneracy theory \citep{greenberg1986analysis} study how the primal and dual vary with the right-hand side (RHS). Simplex implementations report the optimal basis status as a standard output \citep{bixby1992implementing}, making basis reconstruction practical. Yet this rich sensitivity data is not exploited by any existing SO gradient estimator.

\paragraph{Differentiable optimization.}
\citet{amos2017} differentiate through one optimization layer via the Karush--Kuhn--Tucker (KKT) system (see also the survey by \citet{kotary2021end}). \citet{mandi2020interior} solve the LP-based prediction+optimization problem via interior-point differentiation. These methods handle a single optimization layer embedded in a learning pipeline. Our setting requires propagating through multi-period, multi-agent dynamics where the inner optimization appears at every epoch---a fundamentally different challenge.

\paragraph{Equilibrium models and supply chain.}
Mathematical programs with equilibrium constraints (MPEC) and equilibrium models \citep{hobbs2001,metzler2003nash,gabriel2013} and multilevel stochastic programming \citep{shapiro2014} solve the inner layer to a closed-form equilibrium or exactly. Our inner layer is a stochastic, non-equilibrium simulation whose path cannot be pre-computed. \citet{facchinei2003finite} provide the theoretical foundation for complementarity, but their framework does not handle path-dependent stochastic dynamics. On the application side, the bullwhip effect \citep{sterman1989,lee1997,chen2000,dejonckheere2003measuring} is the canonical feedback-amplification phenomenon in multi-echelon supply chains \citep{clark1960optimal,chen1994lower,longo2008,simchilevi2014}; our Theorem~\ref{thm:var} formalizes it as the regime where IPA variance grows exponentially.

\paragraph{Multi-agent reinforcement learning.}
Multi-agent reinforcement learning (MARL) \citep{sutton2018,powell2022} also studies systems of autonomous agents that interact through shared state, but it treats each agent's decision rule as a learned policy whose sensitivity must be re-estimated from simulation or replay data, not as a structured optimization whose basis and dual variables carry exact local derivative information. When the inner problem is an LP or a smooth DP, that geometry is available for free, and a learning-based controller that does not use this structure does not obtain the single-replication, unbiased gradient channel we develop. MARL is nonetheless the natural language when no structured inner optimization is available, and the error-budget characterization (Theorem~\ref{thm:eps}) bridges the two views by casting policy sub-optimality as the same surrogate-error quantity that governs LP-horizon surrogates.

In short, existing work provides either (i)~gradient estimators that do not read the inner optimization's basis, (ii)~equilibrium or multilevel reformulations that do not handle path-dependent stochastic dynamics, or (iii)~learning-based controllers that do not exploit structured inner geometry. The gap we address is the intersection: non-equilibrium, path-dependent, multi-agent systems whose agents solve structured optimizations, and whose outer policy can be optimized with the exact, cheap, and bounded-error approach developed below.

\subsection{Contributions}\label{sec:contrib}

This paper makes four contributions, organized around the question: when an agent-based simulation is a system of optimizers, how should one exploit the inner optimization's geometry to make the outer simulation optimization tractable, accelerated, and bounded in error?

\begin{enumerate}
\item \textbf{An exact, single-replication gradient of the outer objective from the LP basis} (\S\ref{sec:problem}, \S\ref{sec:theory}). We prove that the dual-induced IPA estimator is unbiased on a full-measure set of policies (Theorem~\ref{thm:ipa}), with the derivative--expectation exchange reducing to linearity of expectation for LP inners (Remark~\ref{rmk:poly}). The novelty is propagating the LP basis-sensitivity of \citet{gal1995,bazaraa2010} forward through the multi-period, multi-agent stochastic dynamics, yielding $\nabla J(\theta)$ from one replication; classical IPA \citep{glasserman1991,glasserman1995sensitivity,fu1994sample} reads gradients from event logic or base-stock rules rather than an embedded optimization.

\item \textbf{Online basis-change detection and a CRN covariance bound governed by a basis-disagreement count} (\S\ref{sec:theory}). Basis changes are the only obstructions to IPA unbiasedness; we prove they form a measure-zero set (Theorem~\ref{thm:kink}) that is detected online by comparing active-row signatures, and we derive a CRN covariance bound decomposing paired-difference variance into a within-region term and a basis-disagreement term $\kappa$ (Theorem~\ref{thm:crn}). The diagnostic $\kappa$ is a free by-product of the two forward solves and drives a kink-aware replication allocation that sharpens confidence intervals at fixed budget.

\item \textbf{Exponential feedback amplification of IPA variance, tamed by an error-budgeted surrogate} (\S\ref{sec:theory}, \S\ref{sec:algo}). In coupled systems the IPA variance grows as $\Theta(g^{2m})$ in the feedback depth $m$ (Theorem~\ref{thm:var}), formalizing the bullwhip phenomenon \citep{sterman1989,lee1997,chen2000,dejonckheere2003measuring}; surrogate-as-decision dampens the gain, and its error budget (Theorem~\ref{thm:eps}) satisfies $\varepsilon=0$ iff the state carries no decision-relevant information. The characterization is agnostic to the inner type, unifying LP-horizon and RL/DP-policy surrogates under a single bound.

\item \textbf{A computational framework, \emph{PRIME}, validated on six testbeds including a 1{,}000-SKU supply chain} (\S\ref{sec:algo}--\S\ref{sec:res-r5}). We compose the IPA gradient (C1), adaptive step size (C2), and multi-start ensemble (C3) into a single stochastic-approximation solver, PRIME, and validate it on six testbeds spanning all three stochastic taxonomy classes explored in this paper (I/LP/s, I/DP/s, C/LP/s). On a JD.com supply chain with 1{,}000 stock-keeping units (SKUs) across six distribution centers (one regional distribution center, RDC, and five fulfillment distribution centers, FDCs), PRIME scales to industrial-size SOSO instances, and a head-to-head comparison against classic SO baselines (finite-difference SA, SPSA, stochastic Kriging, Knowledge Gradient) in \S\ref{sec:res-r1} shows PRIME achieves the best or tied-best final gap on all six testbeds with substantially narrower seed-to-seed variance, while black-box gradient methods suffer the exponential variance amplification of Theorem~\ref{thm:var} on feedback-coupled systems.
\end{enumerate}

\subsection{Organization}

Section~\ref{sec:problem} states the modeling assumptions, formalizes the SOSO problem together with its three structural elements, builds a classification from them, and identifies four solution difficulties (D1--D4). Section~\ref{sec:theory} develops the structural analysis, opening with the deterministic baseline (\S\ref{sec:det-case}) and proceeding through Theorems~\ref{thm:ipa}--\ref{thm:eps}, each annotated with its conditions and which difficulty it resolves. Section~\ref{sec:algo} presents algorithms. Section~\ref{sec:expts} reports five key computational results with systematic experimental evidence, including a comparison against classic simulation-optimization baselines and real-world case studies (\S\ref{sec:res-r1}). Section~\ref{sec:disc} discusses scope, when to use IPA versus finite differences, and limitations. Section~\ref{sec:conc} concludes. The proofs of all theorems, the supporting verification and technical lemmas, and a glossary of abbreviations are provided in the online e-companion to this paper.

\section{Problem formulation}\label{sec:problem}

This chapter formalizes the SOSO problem introduced in \S\ref{sec:intro}. We first state the modeling assumptions and their justification (\S\ref{sec:assump}); then give the formal model and define its three structural elements---agent coupling, inner problem type, and exogenous dynamics---in mathematical terms (\S\ref{sec:setting}); then build a classification of SOSO from these elements and state which classes this paper addresses (\S\ref{sec:taxonomy}); and close by identifying the sources of computational difficulty (\S\ref{sec:difficulties}).

\subsection{Modeling assumptions and justification}\label{sec:assump}

We study stochastic multi-agent systems in which each agent solves a structured optimization parameterized by its state $s_{i,t}$ and by a shared policy $\theta$, and the system evolves under exogenous noise $\xi$. Writing the inner problem compactly as $\max_a\,c(s,\theta)^{\top}a$ subject to $Ga\le h(s,\theta)$, $a\ge0$, the objects involved are a technology matrix $G$, a cost vector $c$, a right-hand side $h$, a state transition map $\phi_i$, and a per-period reward $\pi_i$; the formal model is given in \S\ref{sec:setting}. We impose the following assumptions, which define the class of systems we analyze and underpin the sensitivity theory of \S\ref{sec:theory}.

\begin{description}
\item[\textnormal{(A1)} LP structure:] $G$ is constant and $h$ is affine in $(s,\theta)$. This is standard in production planning and network-flow models, where the capacity matrix is fixed and only the right-hand side varies with demand or cost parameters \citep{simchilevi2014,vanderbei2014}.

\item[\textnormal{(A2)} Non-degeneracy and transversality:] each inner LP has a unique optimal basis, and the basis-change residual map has $0$ as a regular value. Non-degeneracy is generic in model parameters \citep{bazaraa2010}; degeneracy, when it occurs, is resolved by a lexicographic perturbation \citep{gal1995}.

\item[\textnormal{(A3)} Generic-no-kink orbit:] on a full-measure set $\Theta^{\circ}\subseteq\Theta$, the orbit stays in the interior of a single LP critical region almost surely. This is a genericity condition that follows from (A2) and the regular value theorem.

\item[\textnormal{(A4)} Smooth, bounded data:] $c$, $\phi_i$, and $\pi_i$ are $C^1$ with bounded derivatives. This is standard in simulation theory; the data of the systems we study are polynomial or piecewise-linear ($C^\infty$ within each critical region).

\item[\textnormal{(A5)} Finite state moments:] $\mathbb{E}\sum_{i,t}\|s_{i,t}\|^2<\infty$. This holds for bounded state spaces and for Lipschitz dynamics driven by light-tailed noise, as verified per testbed in the online e-companion.
\end{description}

\subsection{The SOSO problem and its three structural elements}\label{sec:setting}

We formalize the model from \S\ref{sec:assump} and define its three structural elements mathematically.

\paragraph{Sets, decisions, and the inner optimization.}
The system comprises $N$ agents $i\in\mathcal{I}=\{1,\dots,N\}$ operating over periods $t\in\mathcal{T}=\{0,\dots,T{-}1\}$. At period $t$, agent $i$ observes state $s_{i,t}$ and chooses action $a_{i,t}$ by solving a structured optimization
\begin{equation}\label{eq:inner-gen}
a^{*}_{i,t}\in\arg\max_{a\in\mathcal{X}_{i,t}(s,\theta)}
f_i\!\bigl(a;\,s_{i,t},\theta\bigr),
\end{equation}
whose leading case, under (A1), is the linear program
\begin{equation}\label{eq:innerlp}
a^{*}_{i,t}=\arg\max_a\, c_i^\top a
\quad\text{s.t.}\quad Ga\le h_i,\; a\ge 0,
\end{equation}
with $c_i=c(s_{i,t},\theta)$, $h_i=h(s_{i,t},\theta)$, and constant $G\in\mathbb{R}^{m\times n}$.

\paragraph{State dynamics and objective.}
The state evolves by $s_{i,t+1}=\phi_i(s_t,a^{*}_t,\xi_t,\theta)$, driven by exogenous noise $\xi=(\xi_0,\dots,\xi_{T-1})$ on $(\Omega,\mathcal{F},\mathbb{P})$, which also serves as the common-random-numbers source. The planner chooses $\theta\in\Theta\subseteq\mathbb{R}^d$ to minimize
\begin{equation}\label{eq:outer}
J(\theta)=\mathbb{E}_\xi\!\left[\sum_{t\in\mathcal{T}}\sum_{i\in\mathcal{I}}
\pi_i(a^{*}_{i,t},s_{i,t},\theta)\right].
\end{equation}

\paragraph{The three structural elements.}
The modeling choices above are governed by three elements, defined below in mathematical terms; \S\ref{sec:taxonomy} classifies SOSO by their combinations.

\begin{description}
\item[\textnormal{(E1)} Agent coupling:] is a property of the transition $\phi_i$. Writing $s_t=(s_{1,t},\dots,s_{N,t})$, agents are \emph{independent} when each transition depends only on the agent's own state and action, $s_{i,t+1}=\phi_i(s_{i,t},a_{i,t},\xi_t,\theta)$, equivalently $\partial\phi_i/\partial s_{j,t}\equiv0$ and $\partial\phi_i/\partial a_{j,t}\equiv0$ for $j\ne i$. They are \emph{coupled} when the transition depends on other agents' state or action---equivalently, the cross-agent Jacobian $\partial s_{i,t+1}/\partial s_{j,t}$ is nonzero for some $i\ne j$, as arises through a shared market-clearing price, an upstream inventory, or a network flow.

\item[\textnormal{(E2)} Inner problem type:] is determined by the structure of $(f_i,\mathcal{X}_{i,t})$: an \emph{LP} when $f_i$ is affine in $a$ and $\mathcal{X}_{i,t}=\{a\ge0:Ga\le h_i(s,\theta)\}$ is a polyhedron with constant $G$; a \emph{MILP} when the same holds with a subset of components of $a$ restricted to integers; and a \emph{DP/RL} policy---exemplified by finite Markov decision processes (MDPs) and the continuous linear-quadratic regulator (LQR; the LQ setting)---when the agent maximizes a value function, $a^{*}_{i,t}\in\arg\max_a Q_i(a,s_{i,t})$, with $Q_i$ obtained from Bellman recursion or learned from data.

\item[\textnormal{(E3)} Exogenous dynamics:] is determined by the law $\mathbb{P}$ of $\xi$: \emph{deterministic} when $\xi$ is a fixed sequence ($\mathbb{P}$ a point mass), and \emph{stochastic} otherwise, with sub-cases including i.i.d.\ noise, seasonal noise whose marginal varies periodically in $t$, and autocorrelated (e.g., Markov) noise.
\end{description}

\subsection{Classification of SOSO}\label{sec:taxonomy}\label{sec:variants}

The three structural elements of \S\ref{sec:setting}---coupling (E1), inner problem type (E2), and exogenous dynamics (E3)---define a taxonomy of SOSO instances. Figure~\ref{fig:soso-taxonomy} lays out the resulting classes and highlights those this paper addresses; the exposition follows a simple-to-complex progression throughout ---deterministic before stochastic, independent agents before coupled, LP before MILP/DP---mirroring the order in which the technical complications are introduced.

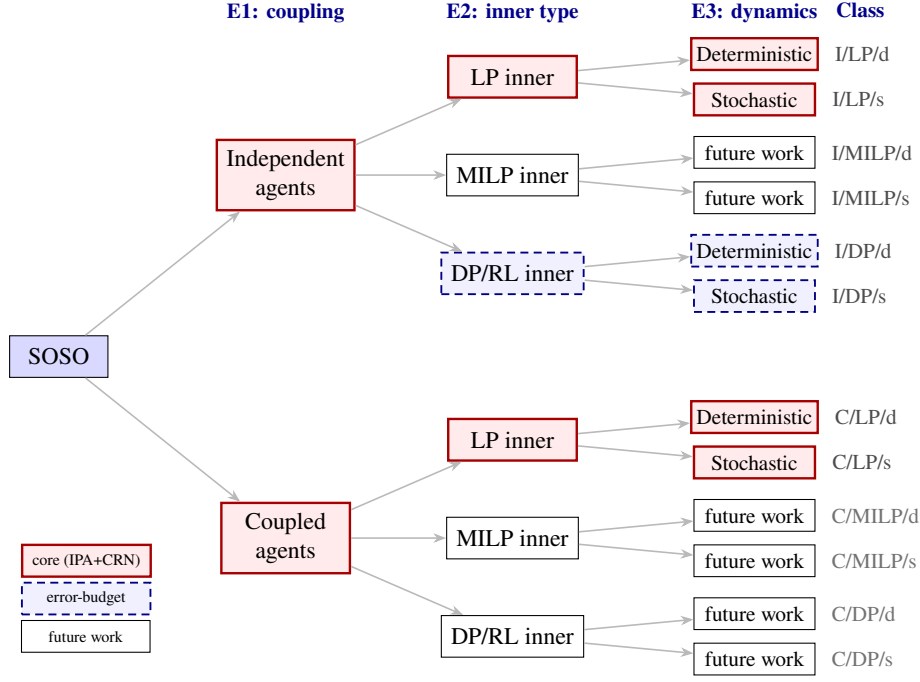
\begin{figure}[t]
\centering
\begin{tikzpicture}[
  font=\footnotesize,
  every node/.style={align=center},
  box/.style={draw, minimum width=1.7cm, minimum height=0.55cm, fill=blue!5},
  boxfw/.style={draw, minimum width=1.7cm, minimum height=0.55cm, fill=white},
  cell/.style={draw, minimum width=1.6cm, minimum height=0.42cm, font=\scriptsize, inner sep=2pt},
  hcore/.style={draw=red!65!black, line width=0.9pt, fill=red!8},
  hbudget/.style={draw=blue!50!black, line width=0.7pt, fill=blue!6, densely dashed},
  arr/.style={-{Stealth[length=1.7mm]}, semithick, gray!55},
  lab/.style={font=\tiny\itshape, inner sep=1pt},
]
\node[box, fill=blue!15, minimum width=1.3cm] (root) at (0,0) {SOSO};

\node[font=\scriptsize\bfseries, text=blue!55!black] at (3.0, 4.55) {E1: coupling};
\node[font=\scriptsize\bfseries, text=blue!55!black] at (6.0, 4.55) {E2: inner type};
\node[font=\scriptsize\bfseries, text=blue!55!black] at (9.2, 4.55) {E3: dynamics};
\node[font=\scriptsize\bfseries, text=blue!55!black] at (10.6, 4.58) {Class};

\node[box, hcore] (ind)  at (3.0, 2.4) {Independent\\agents};
\node[box, hcore] (coup) at (3.0,-2.4) {Coupled\\agents};

\node[box, hcore]   (ind-lp)   at (6.0, 3.7) {LP inner};
\node[boxfw]        (ind-milp) at (6.0, 2.4) {MILP inner};
\node[box, hbudget] (ind-dp)   at (6.0, 1.1) {DP/RL inner};
\node[box, hcore] (coup-lp)   at (6.0,-1.1) {LP inner};
\node[boxfw]      (coup-milp) at (6.0,-2.4) {MILP inner};
\node[boxfw]      (coup-dp)   at (6.0,-3.7) {DP/RL inner};

\node[cell, hcore]   (ind-lp-det)   at (9.2, 4.0) {Deterministic};
\node[cell, hcore]   (ind-lp-sto)   at (9.2, 3.4) {Stochastic};
\node[cell]          (ind-milp-det) at (9.2, 2.7) {\scriptsize future work};
\node[cell]          (ind-milp-sto) at (9.2, 2.1) {\scriptsize future work};
\node[cell, hbudget] (ind-dp-det)   at (9.2, 1.4) {Deterministic};
\node[cell, hbudget] (ind-dp-sto)   at (9.2, 0.8) {Stochastic};

\node[cell, hcore] (coup-lp-det)   at (9.2,-0.8) {Deterministic};
\node[cell, hcore] (coup-lp-sto)   at (9.2,-1.4) {Stochastic};
\node[cell]        (coup-milp-det) at (9.2,-2.1) {\scriptsize future work};
\node[cell]        (coup-milp-sto) at (9.2,-2.7) {\scriptsize future work};
\node[cell]        (coup-dp-det)   at (9.2,-3.4) {\scriptsize future work};
\node[cell]        (coup-dp-sto)   at (9.2,-4.0) {\scriptsize future work};

\node[font=\scriptsize, text=black!70, right=0.08cm of ind-lp-det]   {I/LP/d};
\node[font=\scriptsize, text=black!70, right=0.08cm of ind-lp-sto]   {I/LP/s};
\node[font=\scriptsize, text=black!70, right=0.08cm of ind-milp-det] {I/MILP/d};
\node[font=\scriptsize, text=black!70, right=0.08cm of ind-milp-sto] {I/MILP/s};
\node[font=\scriptsize, text=black!70, right=0.08cm of ind-dp-det]   {I/DP/d};
\node[font=\scriptsize, text=black!70, right=0.08cm of ind-dp-sto]   {I/DP/s};
\node[font=\scriptsize, text=black!70, right=0.08cm of coup-lp-det]  {C/LP/d};
\node[font=\scriptsize, text=black!70, right=0.08cm of coup-lp-sto]  {C/LP/s};
\node[font=\scriptsize, text=black!55, right=0.08cm of coup-milp-det]{C/MILP/d};
\node[font=\scriptsize, text=black!55, right=0.08cm of coup-milp-sto]{C/MILP/s};
\node[font=\scriptsize, text=black!55, right=0.08cm of coup-dp-det]  {C/DP/d};
\node[font=\scriptsize, text=black!55, right=0.08cm of coup-dp-sto]  {C/DP/s};

\draw[arr] (root) -- (ind);
\draw[arr] (root) -- (coup);
\draw[arr] (ind) -- (ind-lp);
\draw[arr] (ind) -- (ind-milp);
\draw[arr] (ind) -- (ind-dp);
\draw[arr] (ind-lp) -- (ind-lp-det);
\draw[arr] (ind-lp) -- (ind-lp-sto);
\draw[arr] (ind-milp) -- (ind-milp-det);
\draw[arr] (ind-milp) -- (ind-milp-sto);
\draw[arr] (ind-dp) -- (ind-dp-det);
\draw[arr] (ind-dp) -- (ind-dp-sto);
\draw[arr] (coup) -- (coup-lp);
\draw[arr] (coup) -- (coup-milp);
\draw[arr] (coup) -- (coup-dp);
\draw[arr] (coup-lp) -- (coup-lp-det);
\draw[arr] (coup-lp) -- (coup-lp-sto);
\draw[arr] (coup-milp) -- (coup-milp-det);
\draw[arr] (coup-milp) -- (coup-milp-sto);
\draw[arr] (coup-dp) -- (coup-dp-det);
\draw[arr] (coup-dp) -- (coup-dp-sto);

\node[cell, hcore,   minimum width=1.7cm, font=\tiny, anchor=west] at (-0.5,-2.7) {core (IPA+CRN)};
\node[cell, hbudget, minimum width=1.7cm, font=\tiny, anchor=west] at (-0.5,-3.2) {error-budget};
\node[cell,          minimum width=1.7cm, font=\tiny, anchor=west] at (-0.5,-3.7) {future work};
\end{tikzpicture}
\caption{A taxonomy of SOSO along the three structural elements of \S\ref{sec:setting}: agent coupling (independent vs.\ coupled), inner problem type (LP, MILP, DP/RL), and exogenous dynamics (deterministic vs.\ stochastic). Solid red outline: core of this paper, where the full IPA/CRN analysis applies. Dashed blue: partially explored via the error-budget and surrogate-as-decision framework. Plain outline: left for future research. The deterministic case is the analytic baseline; stochasticity is the regime where simulation becomes necessary.}\label{fig:soso-taxonomy}
\end{figure}

\paragraph{The classes this paper addresses.}
We tag each SOSO class by a label of the form \emph{coupling/inner/dynamics}: $I$ or $C$ for independent or coupled agents, LP/MILP/DP for the inner problem type, and $d$ or $s$ for deterministic or stochastic dynamics. Thus $C/LP/d$ denotes coupled agents with an LP inner under deterministic dynamics. The same tag annotates the leaves of Figure~\ref{fig:soso-taxonomy}, in the spirit of Kendall's notation for queueing systems. The computational experiments in \S\ref{sec:expts} address six testbeds drawn from the deterministic baseline and three stochastic classes, progressing from simplest to most challenging:
\begin{enumerate}
\item \textbf{Deterministic, any inner, any coupling} ($I/LP/d$, $C/LP/d$; \S\ref{sec:det-case}). With known $\xi$, there is no sampling error: $J(\theta)$ is a deterministic parametric program, the IPA gradient is exact from the LP basis, and surrogate errors are computable rather than estimated. This is the analytic baseline against which the stochastic complications are measured. Testbeds TB1det, TB2det, and TB5det (Table~\ref{tab:testbeds}) cover this case.

\item \textbf{Stochastic, independent-agent LP} ($I/LP/s$, TB5). On real-seasonal SupplyGraph demand, PRIME achieves $0.004\pm0.004$, confirming that C1 and C2 extend to single-agent LPs without coupling. The $I/LP/s$ cell lies in the core (red in Figure~\ref{fig:soso-taxonomy}).

\item \textbf{Stochastic, independent-agent DP/RL} ($I/DP/s$, TB3 and TB6). Two testbeds span this cell. TB3 (continuous LQR) confirms C1's IPA channel extends beyond LP to smooth DP inners; TB6 (discrete MDP) tests C3 in the absence of IPA, where multi-start spatial diversification provides the decisive advantage. The $I/DP/s$ cell is partially explored (dashed blue).

\item \textbf{Stochastic, coupled-agent LP} ($C/LP/s$, TB1, TB2, TB4). Three testbeds span the core cell: TB1 (minimal market) validates all three components under mild feedback; TB2 (Beer Game) stresses C2 and C3 under strong feedback amplification ($g{=}2$, $m{=}4$); TB4 (JD.com chain) demonstrates industrial-scale performance (1{,}000 SKUs). The $C/LP/s$ cell lies in the core (red).
\end{enumerate}
Cells marked as future work in Figure~\ref{fig:soso-taxonomy}---$I/MILP/d$, $I/MILP/s$, $C/MILP/s$, and $C/DP/s$---are important open problems: the error-budget framework applies in principle, but discrete decisions introduce new challenges (non-convex outer objective, combinatorial sensitivity) beyond this paper's scope. The $C/MILP/s$ and $C/DP/s$ cells are particularly challenging because coupled feedback through integer decisions amplifies discretely rather than geometrically.

\subsection{Sources of difficulty}\label{sec:difficulties}

We close by identifying four difficulties that any solution to the SOSO problem must address. Each arises from one or more of the three structural elements of \S\ref{sec:setting}---agent coupling (E1), inner problem type (E2), and exogenous dynamics (E3)---and each is labeled (D1--D4) for cross-reference with the theory and algorithms.

\begin{description}
\item[\textnormal{(D1)} Nested cost (all three elements):] Each replication of $J(\theta)$ requires $N\cdot T$ inner solves; an SO procedure examining $K$ candidates with $M$ replications each calls the solver $N\cdot T\cdot M\cdot K$ times. The total scales with the coupling breadth $N$ (E1), with the inner problem type (E2)---an LP solve is cheap, but a MILP or a full DP backup is orders of magnitude slower---and with the number of replications $M$ that stochastic dynamics (E3) demand. For a five-firm, 30-period case study ($N{=}5$, $T{=}30$, $M{=}200$, $K{=}11$), this is $330{,}000$ LP solves for a single $J(\theta)$ curve scan; at industrial scale ($N{=}50$, $T{=}365$, $M{=}1000$, $K{=}50$), the total reaches $\sim10^9$ solves. Even at $0.1$\,s per solve, a single central-FD gradient evaluation for the case-study scale above ($2MNT=6{\times}10^4$ solves) takes $\sim 2$ hours; at the industrial scale the same evaluation would require roughly $3.65\times10^7$ solves, or about $1{,}000$ hours.

\item[\textnormal{(D2)} Wasted structure (inner type, E2):] The inner optimization's geometry---the LP basis and dual variables, or more generally the active set and value-function curvature---carries exact local derivative information, but black-box SO via finite differences ignores it, requiring two replications per gradient point with variance inflated by the $1/h^2$ factor inherent in the differencing step. The severity depends on the inner problem type: LP inners have the richest exploitable structure, whereas MILP and discrete-DP inners have ill-defined active sets.

\item[\textnormal{(D3)} Feedback amplification (coupling, E1):] This difficulty is present only when agents are coupled: one agent's action enters another agent's state, so the sensitivity propagated by IPA is amplified at each feedback stage, with variance growing as $\Theta(g^{2m})$ in the chain depth $m$ ($g$ the per-stage gain; $2^{8}=256$ for $g{=}2$, $m{=}4$). It is absent for independent agents, whose sensitivity stays bounded.

\item[\textnormal{(D4)} Error control (inner type and dynamics, E2/E3):] Replacing an exact inner solve by a cheaper surrogate saves compute but introduces a decision error $\varepsilon$ that shifts the outer optimum. The magnitude of this error is governed jointly by the inner problem type and the exogenous dynamics: under i.i.d.\ dynamics a static surrogate is often near-optimal, whereas seasonal dynamics make the surrogate miss exploitable temporal structure. Without a theory of when this shift is bounded, the planner cannot safely trade accuracy for speed.
\end{description}

\section{Structural analysis and theory}\label{sec:theory}

This section develops the structural properties on which the algorithms rest. We begin in \S\ref{sec:det-case} with the deterministic special case, which is itself part of the structural analysis and isolates the parametric geometry from the stochastic complications. The remaining results form a single logical chain: parametric-LP critical regions give the local action sensitivity; forward propagation gives the IPA gradient; measure-zero basis changes are the only obstructions; a CRN decomposition quantifies the variance-reduction loss at those obstructions; feedback amplification identifies when IPA loses to finite differences; and the error budget characterizes when surrogate decisions are safe. Each subsection is annotated with the difficulty it addresses.

\subsection{The deterministic case: an analytic baseline}\label{sec:det-case}

We start with the deterministic special case, which serves as the analytic baseline. When the exogenous process $\xi$ is a known sequence (degenerate variance), the outer objective $J(\theta)=L(\theta,\xi)$ is a deterministic parametric program and simulation is unnecessary. This case is not merely academic---it appears in practice when a planner evaluates a policy against \emph{historical} demand (a single realized trajectory) or when the LP right-hand side is a known function of the state and $\theta$. In this zero-variance limit the three stochastic mechanisms on which the rest of the theory rests---gradient estimation, variance reduction, and error control---each collapse to a trivial form: the IPA gradient becomes exact, the common-random-numbers apparatus becomes vacuous, and the surrogate error becomes computable rather than estimated. In this zero-variance limit the deterministic baseline isolates the parametric geometry---kinks, critical regions, and the surrogate-versus-exact tradeoff---from the sampling structure that motivates the rest of the theory.

\subsection{Critical-region structure and IPA unbiasedness}\label{sec:ipa}

By parametric-LP theory \citep{gal1995,bazaraa2010}, the $(s,\theta)$-space is partitioned into critical regions $\{P_k\}$ on each of which the optimal basis $B_k$ is constant and the primal is the affine map $a^{*}=B_k^{-1}h_k$. The local sensitivity
\begin{equation}\label{eq:lpsens}
\partial a^{*}/\partial\theta=B_k^{-1}\,\partial h_k/\partial\theta
\end{equation}
is one matrix--vector product. Forward-propagating the state sensitivity $S_{i,t}:=\partial s_{i,t}/\partial\theta$ via the chain rule gives the IPA estimator $\widehat{J}'_{\mathrm{IPA}}(\theta,\xi)$ from the forward-sensitivity recursions in Algorithm~\ref{alg:framework} (Section~\ref{sec:algo}).

\begin{theorem}[IPA unbiasedness]\label{thm:ipa}
Under (A1)--(A5), for every $\theta\in\Theta^{\circ}$, $\mathbb{E}_\xi[\widehat{J}'_{\mathrm{IPA}}]=\nabla J(\theta)$.
\end{theorem}

The assumptions are mild: (A1)--(A4) are standard in LP and simulation theory, and (A5) is verified per testbed in \S\ref{sec:expt-design}. The key point is that Theorem~\ref{thm:ipa} reads the gradient from the LP basis at no additional simulation cost, directly resolving the wasted-structure difficulty (D2): whereas finite differences require two replications and introduce $O(h^2)$ bias, IPA provides an exact, single-replication gradient from the shadow price that the LP solver already computes.

Figure~\ref{fig:fd_vs_ipa} contrasts the FD and IPA approaches. FD treats the simulation as a black box, requiring two replications and producing a noisy, $O(h^2)$-biased estimate. IPA reads the shadow price from the LP solver---at no additional cost---and propagates this sensitivity forward through the dynamics to obtain an exact, single-replication gradient.

\begin{figure}[H]
\centering
\begin{tikzpicture}[>=Stealth, font=\small, scale=0.9, transform shape,
  box/.style={draw, minimum height=0.75cm, minimum width=1.4cm, align=center, fill=blue!8, thick},
  simbox/.style={draw, minimum height=0.95cm, minimum width=2.0cm, align=center, fill=blue!14, thick},
  lpbox/.style={draw, minimum height=0.95cm, minimum width=2.3cm, align=center, fill=green!15, thick},
  gradbox/.style={draw, minimum height=0.75cm, minimum width=2.3cm, align=center, fill=green!25, thick},
  opbox/.style={draw, minimum height=0.75cm, minimum width=2.4cm, align=center, fill=orange!15, thick},
  blackbox/.style={draw, dashed, thick, gray, minimum height=1.5cm, minimum width=4.5cm},
  arr/.style={->, thick, >=Stealth},
  labelbad/.style={font=\scriptsize\bfseries, text=white, fill=red!70!black, inner sep=2.5pt},
  labelgood/.style={font=\scriptsize\bfseries, text=white, fill=green!60!black, inner sep=2.5pt}
]
\draw[gray!40, very thick, dashed] (2.2,-2.1) -- (2.2,4.3);

\node[font=\small\bfseries, text=blue!70!black] at (-1.2, 4.1) {Finite Differences};
\node[box] (th) at (-2.3, 3.1) {$\theta{-}h$};
\node[box] (thh) at (-0.1, 3.1) {$\theta{+}h$};
\node[simbox] (sim1) at (-2.3, 1.75) {Simulation\\run 1};
\node[simbox] (sim2) at (-0.1, 1.75) {Simulation\\run 2};
\node[blackbox] at (-1.2, 1.75) {};
\node[font=\scriptsize, text=black, rotate=90] at (-3.6, 1.75) {black box};
\node (l1) at (-2.3, 0.5) {$L(\theta{-}h)$};
\node (l2) at (-0.1, 0.5) {$L(\theta{+}h)$};
\node[opbox] (diff) at (-1.2, -0.7) {$\frac{L(\theta{+}h)-L(\theta{-}h)}{2h}$};
\node[labelbad] at (-1.2, -1.7) {noisy \;\textbullet\; $O(h^2)$ bias \;\textbullet\; $2\times$ cost};

\draw[arr] (th) -- (sim1);
\draw[arr] (thh) -- (sim2);
\draw[arr] (sim1) -- (l1);
\draw[arr] (sim2) -- (l2);
\draw[arr] (l1) -- (diff);
\draw[arr] (l2) -- (diff);

\node[font=\small\bfseries, text=green!50!black] at (6.2, 4.1) {IPA (exploits LP geometry)};
\node[box] (thi) at (5.1, 3.1) {$\theta$};
\node[simbox] (simi) at (5.1, 1.75) {Single\\simulation run};
\node[lpbox] (lp) at (5.1, 0.25) {LP solver $\rightarrow$ dual $\lambda$};
\node[gradbox] (grad) at (5.1, -1.1) {$\widehat{J}'_{\mathrm{IPA}}$};
\node[labelgood] at (5.1, -2.0) {exact \;\textbullet\; unbiased \;\textbullet\; $1\times$ cost};
\node[font=\scriptsize, text=green!50!black, align=left] at (8.0, 0.25) {$\lambda = B_k^{-1}\partial h/\partial\theta$};

\draw[arr] (thi) -- (simi);
\draw[arr] (simi) -- (lp);
\draw[arr] (lp) -- (grad);
\draw[arr, green!50!black, thick] (lp.east) -- (7.1,0.25);
\end{tikzpicture}
\caption{Conceptual comparison of gradient estimation for the SOSO problem. \emph{Left:} Finite differences treat the simulation as a black box, run it twice at $\theta{-}h$ and $\theta{+}h$, and difference the outputs---producing a noisy, $O(h^2)$-biased estimate at twice the cost. \emph{Right:} IPA runs the simulation once and reads the LP's dual variable (shadow price $\lambda=B_k^{-1}\partial h/\partial\theta$), then propagates the sensitivity forward to obtain an exact, unbiased gradient at the cost of one replication.}
\label{fig:fd_vs_ipa}
\end{figure}
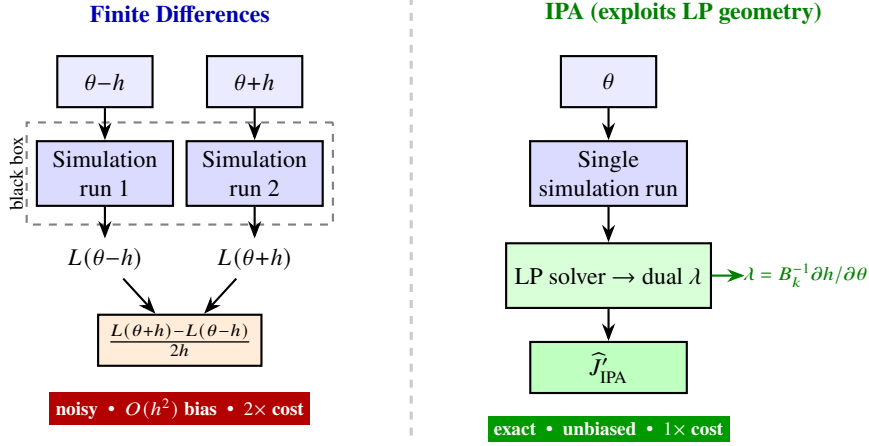

Figure~\ref{fig:basis_change} illustrates the critical-region structure using a 2-variable LP. The feasible region is a convex polytope; as $\theta$ varies, the objective direction rotates, and the optimal vertex moves along the hull. At a vertex switch (edge crossing), the basis changes---this is the kink that Theorem~\ref{thm:kink} identifies as a measure-zero event.

\begin{figure}[H]
\centering
\begin{tikzpicture}[>=Stealth, font=\small, scale=0.8, transform shape]

\fill[gray!10] (1,1) -- (5,1) -- (5.5,3) -- (3,4.5) -- (1,3) -- cycle;
\draw[very thick, gray!70!black] (1,1) -- (5,1) -- (5.5,3) -- (3,4.5) -- (1,3) -- cycle;

\foreach \x/\y in {1/1, 5/1, 5.5/3, 3/4.5, 1/3} {
  \fill[gray!50!black] (\x,\y) circle (3pt);
}
\fill[blue!70!black]  (5,1) circle (4pt);
\fill[green!60!black] (5.5,3) circle (4pt);
\node[below left=0.8mm and 0.8mm, font=\scriptsize\bfseries] at (1,1) {$v_1$};
\node[below=1.5mm, font=\scriptsize\bfseries, align=center] at (5,1) {$v_2$ (basis $B_1$)};
\node[right=2.0mm, font=\scriptsize\bfseries, align=center] at (5.5,3) {$v_3$ (basis $B_2$)};
\node[above=1mm, font=\scriptsize\bfseries] at (3,4.5) {$v_4$};
\node[left=1mm, font=\scriptsize\bfseries] at (1,3) {$v_5$};

\draw[->, >=Stealth, line width=2pt, red, dashed] (5,1) -- (5.5,3);
\node[red!80!black, font=\scriptsize\bfseries, fill=white, inner sep=2pt, align=center] at (6.6,2.1) {basis change\\$B_1\to B_2$};

\draw[->, thick] (0,0) -- (7,0) node[right] {$x_1$};
\draw[->, thick] (0,0) -- (0,5) node[above] {$x_2$};
\node[font=\small\bfseries, text=gray!45!black] at (3,5.4) {action space $x$};

\draw[dashed, thick, gray!55] (7.7,-0.6) -- (7.7,5.7);

\fill[blue!14]  (8.2,0.5) rectangle (11.1,4.7);
\fill[green!16] (11.1,0.5) rectangle (14.0,4.7);
\draw[thick, gray!70!black] (8.2,0.5) rectangle (14.0,4.7);
\draw[line width=1.5pt, red, dashed] (11.1,0.5) -- (11.1,4.7);
\fill[blue!70!black] (9.6,2.6) circle (4pt);
\node[font=\small, blue!70!black] at (9.6,2.1) {$\theta_1$};
\fill[green!60!black] (12.6,2.6) circle (4pt);
\node[font=\small, green!60!black] at (12.6,2.1) {$\theta_2$};
\node[font=\small, blue!70!black, align=center] at (9.6,3.9) {critical region $P_1$\\(basis $B_1$)};
\node[font=\small, green!60!black, align=center] at (12.6,3.9) {critical region $P_2$\\(basis $B_2$)};
\node[font=\small, red!80!black] at (11.1,0.15) {$\partial P_k$};
\node[font=\small\bfseries, text=gray!45!black] at (11.1,5.4) {policy space $\Theta$};
\end{tikzpicture}
\caption{Critical-region structure of a 2-variable parametric LP, shown in two spaces separated by the dashed divider. \emph{Left (action space $x$):} the feasible polytope has vertices $v_1$--$v_5$; as $\theta$ varies the optimum switches from $v_2$ (basis $B_1$, blue) to $v_3$ (basis $B_2$, green) along the dashed red edge---the basis change, a measure-zero event (Theorem~\ref{thm:kink}). \emph{Right (policy space $\Theta$):} the same switch as two critical regions $P_1$/$P_2$ separated by $\partial P_k$, with $\theta_1\in P_1$ and $\theta_2\in P_2$. The IPA derivative is exact within each region and one-sided at the boundary, where the basis-disagreement count $\kappa$ detects the transition.}\label{fig:basis_change}
\end{figure}
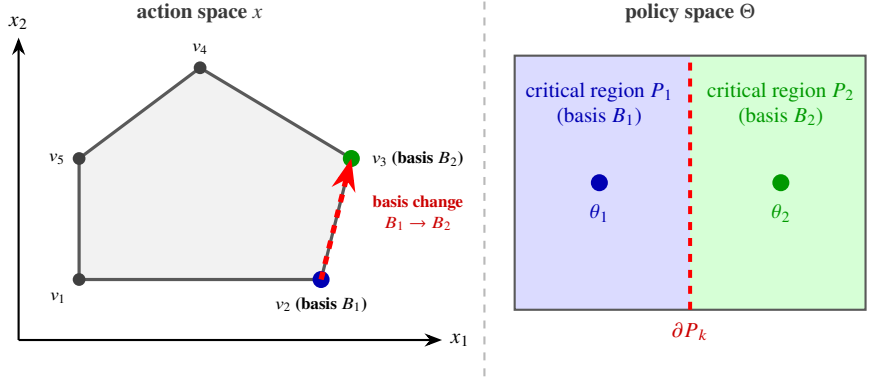

\begin{remark}[Polynomial simplification]\label{rmk:poly}
For LP-inner models, $L(\theta,\xi)$ is piecewise-polynomial in $\theta$ within each critical region, so $\mathbb{E}[\partial L/\partial\theta] =\partial\mathbb{E}[L]/\partial\theta$ reduces to linearity of expectation for polynomials---no measure-theoretic dominated convergence theorem (DCT) is required.
\end{remark}

\subsection{Detectable, measure-zero critical events}\label{sec:kink}

\begin{theorem}[Critical events are measure-zero]\label{thm:kink}
Under (A2)--(A4), $\Theta\setminus\Theta^{\circ}$ has Lebesgue measure zero.
\end{theorem}

Because $0$ is a regular value of the residual map $f_\xi$ by (A2), the \emph{regular value theorem} (a corollary of the implicit function theorem) gives that $f_\xi^{-1}(\{0\})$ is a smooth $(d{-}M)$-submanifold of $\Theta$. Since $d\le M$ in our setting (a low-dimensional policy parameter and many residual components), this submanifold has dimension at most zero and hence Lebesgue measure zero. Only $C^1$ regularity---already guaranteed by (A4)---is required. (For the atypical case $d>M$, one would instead invoke Sard's theorem to control the set of critical values, but the measure-zero conclusion for a fixed regular value $0$ follows already from (A2).) Theorem~\ref{thm:kink} thus ensures IPA is unbiased except on a negligible set, and because basis changes are detected online by comparing active-row signatures, this naturally leads to the CRN diagnostic of Theorem~\ref{thm:crn}.

\subsection{CRN covariance decomposition}\label{sec:crn}

For paired differences $D=L(\theta_a,\xi)-L(\theta_b,\xi)$ with CRN:

\begin{theorem}[CRN covariance decomposition]\label{thm:crn}
$\mathrm{Var}(D)\le 2\bigl(\mathrm{Var}(D)|_{\kappa=0}+\bar{\Delta}^2\,\mathbb{E}[\kappa^2]\bigr)$, where $\kappa$ counts basis-disagreement cells and $\bar{\Delta}$ bounds reward jumps. The universal factor $2$ absorbs the smooth/kink cross-covariance and does not affect the qualitative conclusion: CRN helps within a region ($\kappa=0$) and degrades with $\kappa$.
\end{theorem}

The only assumption beyond (A4) is that $\bar{\Delta}$ is finite, which follows from the reward's Lipschitz continuity and the bounded primal diameter---both consequences of (A1) and (A4). The bound is stated in the finite-second-moment regime $\mathrm{Var}(D)<\infty$ (the case of practical interest; when $\mathrm{Var}(D)=\infty$ the inequality is vacuous), which is the natural condition for the factor-$2$ combining step $\mathrm{Var}(X{+}Y)\le 2\mathrm{Var}(X)+2\mathrm{Var}(Y)$. No additional structure is needed. In practice $\kappa$ is computed by running the two CRN-paired forward solves at $\theta_a$ and $\theta_b$, recording each agent's active-row signature (the set of tight constraints, equivalently the optimal basis) at each $(i,t)$, and counting the cells where the two signatures differ; it is therefore a free by-product of the two solves rather than an extra estimation step. This $\kappa$ diagnostic drives a kink-aware replication allocation (Algorithm~\ref{alg:framework}) that trims the worst-case confidence-interval half-width at fixed budget (Table~\ref{tab:kink-aware}), directly attacking the nested-cost difficulty by making each replication more informative.

\subsection{Variance amplification through feedback}\label{sec:var}

\begin{theorem}[Variance amplification]\label{thm:var}
If the state-sensitivity recursion $S_{t+1}=A_tS_t+b_t$ has per-stage gain $\|A_t\|\le g$ with $g>1$ over $m$ feedback stages, with non-cancelling coherent gains (the scalar case $A_t\ge g>0$ suffices) and integrable shocks, then the IPA-gradient second moment satisfies $\mathbb{E}_\xi|\widehat{J}'_{\mathrm{IPA}}|^2=\Theta(g^{2m})$ in the feedback depth $m$; consequently $\mathrm{Var}[\widehat{J}'_{\mathrm{IPA}}]=\Theta(g^{2m})$ whenever the mean $\nabla J(\theta)$ does not cancel it (e.g., away from a stationary point of $J$). The CRN-FD variance grows only polynomially in $m$.
\end{theorem}

Unlike the previous three theorems, the conditions here are structural rather than mild: the result requires non-cancelling positive gains, as arise naturally from a reactive forecast in supply chains---the bullwhip condition \citep{sterman1989,lee1997,chen2000}. This is precisely the regime where the feedback-amplification difficulty (D3) becomes operative: IPA loses to finite differences because its per-replication variance grows exponentially. The remedy, developed next, is surrogate-as-decision, which dampens the feedback gain.

\subsection{Error budget for surrogate-as-decision}\label{sec:eps}

Replace $a^{*}_{\mathrm{ex}}(s)$ by a surrogate $a^{*}_{\mathrm{su}}(s)$; decision error $\varepsilon=\mathbb{E}_s[\pi(a^{*}_{\mathrm{ex}})-\pi(a^{*}_{\mathrm{su}})]$.

\begin{theorem}[Error-budget characterization]\label{thm:eps}
$\varepsilon\ge0$, and for a static surrogate $\varepsilon=0$ iff the optimal action is a.s.\ state-independent. Agnostic to inner type (LP, MILP, DP/RL).
\end{theorem}

The conditions are mild: no assumptions beyond LP optimality are needed. This theorem addresses two difficulties simultaneously. For D4 (error control), it characterizes exactly when an approximate inner decision is safe: the surrogate introduces no outer-optimum shift precisely when the optimal action is state-independent almost surely. Independent, stationary demand is one common way this arises---when the optimal decision depends only on constant demand statistics---but the characterization is state-independence itself, not stationarity per se. For D3 (feedback amplification), a damped surrogate reduces the per-stage gain $g$ in Theorem~\ref{thm:var}'s bound, attenuating the exponential variance growth. The regime $\varepsilon\approx0$ (observed under iid demand in Section~\ref{sec:res-r5}) is a direct and practically important consequence.

\begin{remark}[Outer-optimum shift is bounded by $\varepsilon$]\label{rmk:delta}
A corollary used throughout the paper: if the surrogate's objective error is uniformly bounded, $|J(\theta,a^{*}_{\mathrm{ex}})-J(\theta,a^{*}_{\mathrm{su}})|\le\varepsilon_{\max}$ for all $\theta$, then adopting the surrogate's optimum costs at most $\delta:=J(\theta^{*}_{\mathrm{su}})-J(\theta^{*}_{\mathrm{ex}})\le2\varepsilon_{\max}$ in true objective ($\varepsilon$-suboptimality). When the surrogate is one-sided optimistic---as an LP relaxation of a minimization is, with $J_{\mathrm{su}}\le J_{\mathrm{ex}}$ and $\varepsilon=J_{\mathrm{ex}}-J_{\mathrm{su}}\ge0$---the bound tightens to $\delta\le\varepsilon_{\max}$.
\end{remark}

\section{Algorithms}\label{sec:algo}

The framework exploits three independently switchable components, each tied to a theoretical result and to one or more of the difficulties D1--D4 (Table~\ref{tab:framework}). We call the resulting solver \emph{PRIME} (Primal-IPA with Robust Multi-start \& Ensemble): a stochastic-approximation outer loop whose gradient is the single-replication, dual-induced IPA estimator of Theorem~\ref{thm:ipa}, with adaptive step-size normalisation and a multi-start ensemble that exploits $\kappa$-diagnosed basis-change synchronization to escape suboptimal attractors. Algorithm~\ref{alg:framework} presents the structure: Step~1 runs $K$ independent SA chains from equispaced starts, Step~2 applies the IPA gradient with adaptive stepping round-robin across chains, and both share one replication engine---the \textsc{ForwardSens} sub-procedure. The subsections develop the three components: \textsc{ForwardSens} (\S\ref{sec:alg-ipa}), adaptive stepping (\S\ref{sec:alg-crn}), and the multi-start ensemble (\S\ref{sec:alg-surrogate}).

\begin{table}[H]
\caption{PRIME components. Each component is labeled C1--C3, named, mapped to its theoretical guarantee, and linked to the difficulty(ies) it resolves. The ablation column gives the degraded form used in \S\ref{sec:expt-design} to isolate the component's marginal contribution.}\label{tab:framework}
\small
\centering
\begin{tabular}{@{}llp{0.42\linewidth}ll@{}}
\toprule
Cmp. & Name & Operation / ablation degraded form & Theory & Difficulty \\
\midrule
C1 & IPA gradient & Reads $\nabla J$ from the inner LP basis; one replication per gradient. Ablation $-C1$ replaces it with CRN-FD (2 replications/gradient). & Thms.~\ref{thm:ipa}, \ref{thm:kink} & D1, D2 \\
C2 & Adaptive step size & EWMA-normalised gradient magnitude makes the step scale-invariant. Ablation $-C2$ uses fixed step size. & Thm.~\ref{thm:var} (scale) & D3 \\
C3 & Multi-start ensemble & $K{=}3$ independent IPA chains from equispaced starts escape suboptimal attractors. Ablation $-C3$ uses single chain. & Thms.~\ref{thm:kink}, \ref{thm:crn} (via $\kappa$) & D3 \\
\bottomrule
\end{tabular}
\\[4pt]{\footnotesize\raggedright D4 (error control) is addressed by Theorem~\ref{thm:eps} and the surrogate-as-decision error budget (\S\ref{sec:eps}), a complementary lever of the framework rather than a PRIME component. D1 (nested cost) is resolved by C1's single-replication IPA gradient, which halves the per-gradient simulation cost and eliminates the $O(h^2)$ differencing overhead of black-box FD. D2 (wasted structure) is specific to C1: only the IPA gradient directly exploits the inner LP's basis and dual variables. D3 (feedback amplification) is attacked from two complementary angles: C2 normalises the gradient-scale landscape that feedback amplification creates (\S\ref{sec:alg-crn}), and C3 provides spatial diversification to escape suboptimal attractors that arise in feedback-coupled landscapes (\S\ref{sec:alg-surrogate}).\par}
\end{table}

\begin{algorithm}[H]
\caption{PRIME, the proposed simulation-optimization solver. The three components C1--C3 from Table~\ref{tab:framework} are mapped to the algorithmic steps: step~1 implements C3 (multi-start), the outer loop in step~2 runs C1 (IPA gradient) with C2 (adaptive step).}
\label{alg:framework}
\small
\begin{algorithmic}[1]
\Require Policy box $[\theta_{\mathrm{lo}},\theta_{\mathrm{hi}}]$, budget $B$ (inner solves).
\Ensure Best validated $\hat\theta^{*}$, IPA gradient $\widehat{J}'_{\mathrm{IPA}}$.
\Statex \textbf{Step 1.} Multi-start initialisation \quad(C3; \S\ref{sec:alg-surrogate})
\State Generate $K{=}3$ equispaced initial points $\{\theta^{(1)}_0,\dots,\theta^{(K)}_0\}$ across $[\theta_{\mathrm{lo}},\theta_{\mathrm{hi}}]$. \Comment{C3: multi-start}
\For{each chain $k=1,\dots,K$}
  \State Initialise $\theta^{(k)}\gets\theta^{(k)}_0$;\; $g^{(k)}_{\mathrm{ewma}}\gets0$.
\EndFor
\Statex \textbf{Step 2.} IPA-driven stochastic approximation with adaptive step \quad(C1+C2; \S\ref{sec:alg-ipa}, \S\ref{sec:alg-crn})
\For{round $r=1,\dots,R$}
  \For{each chain $k=1,\dots,K$}
    \State Draw $\xi_k$;\; call \Call{ForwardSens}{$\theta^{(k)},\xi_k$} \Comment{C1: single-rep IPA}
    \State $g(\theta^{(k)})\gets \widehat{J}'_{\mathrm{IPA}}$;\; $n_{\mathrm{solves}}\gets n_{\mathrm{solves}}+1$.
    \State $g^{(k)}_{\mathrm{ewma}}\gets\beta_g\,g^{(k)}_{\mathrm{ewma}}+(1-\beta_g)\,|g(\theta^{(k)})|$ \Comment{C2: adaptive step}
    \State $d_k\gets g(\theta^{(k)})\,/\,(g^{(k)}_{\mathrm{ewma}}+\varepsilon)$
    \State $\theta^{(k)}\gets \mathrm{clip}(\theta^{(k)}-a_r\cdot d_k,\;\theta_{\mathrm{lo}},\theta_{\mathrm{hi}})$
  \EndFor
  \If{$r\bmod r_{\mathrm{val}}=0$}
    \State Validate each $\theta^{(k)}$ with high-precision $F_{\mathrm{true}}$; track $\hat\theta^{*}$.
  \EndIf
\EndFor
\State \Return $\hat\theta^{*}$.
\Statex \rule{0.9\linewidth}{0.4pt}
\Statex \textbf{Sub-procedure} \Call{ForwardSens}{$\theta,\xi$} \quad(C1; \S\ref{sec:alg-ipa}; Theorem~\ref{thm:ipa})
\State $S_{i,0}\gets 0$;\; $L\gets 0$;\; $\widehat{J}'\gets 0$.
\For{$t=0,\dots,T{-}1$}
  \For{$i=1,\dots,N$}
    \State solve \eqref{eq:innerlp}; read $a^{*}_{i,t}$, $B_k^{-1}$, and basis signature.
    \State $\partial_\theta a^{*}_{i,t}\gets B_k^{-1}(\partial_\theta h_k+\partial_s h_k\,S_{i,t})$
    \State $S_{i,t+1}\gets\partial_s\phi_i\,S_{i,t}+\partial_a\phi_i\,\partial_\theta a^{*}_{i,t}+\partial_\theta\phi_i$
    \State $L\mathrel{+}=\pi_i$;\; $\widehat{J}'\mathrel{+}=\partial_a\pi_i\,\partial_\theta a^{*}_{i,t}+\partial_s\pi_i\,S_{i,t}+\partial_\theta\pi_i$
  \EndFor
\EndFor
\State \Return $L$, $\widehat{J}'_{\mathrm{IPA}}$, $\{B_{i,t}(\theta)\}$.
\end{algorithmic}
\end{algorithm}

\subsection{C1: IPA forward sensitivity}\label{sec:alg-ipa}

The \textsc{ForwardSens} sub-procedure of Algorithm~\ref{alg:framework} runs one replication while propagating the state sensitivity $S_{i,t}:=\partial s_{i,t}/\partial\theta$ forward through the dynamics. At each period and agent, the inner LP is solved once (as in a standard simulation), and the basis inverse $B_k^{-1}$ is read from the solution. The action sensitivity $\partial_\theta a^{*}_{i,t}=B_k^{-1}(\partial_\theta h_k+\partial_s h_k\,S_{i,t})$ requires one matrix--vector multiply; the state update $S_{i,t+1}=\partial_s\phi_i\,S_{i,t}+\partial_a\phi_i\,\partial_\theta a^{*}_{i,t} +\partial_\theta\phi_i$ requires three Jacobian--vector products. The total per-period overhead is $O(N\cdot n^2)$, where $n$ is the action dimension---negligible compared to the LP solve itself ($O(n^3)$ or worse). This gives an unbiased gradient from a single replication, halving the per-gradient cost of central FD and eliminating the $O(h^2)$ bias and $\mathrm{Var}(L)/h^2$ variance inflation.

\paragraph{Basis reconstruction from a real solver.}
For inner LPs too large for vertex enumeration, $B_k^{-1}$ is reconstructed from the simplex solver's reported variable basis statuses (Gurobi's \texttt{VBasis} or HiGHS's basis output). The $m$ basic variables select $m$ columns of the equality matrix $[G|I]$ (slacks included), giving $B_k$ by a single $m\times m$ inversion. Model reuse and warm-starting from the previous period's basis cut per-solve cost roughly tenfold, as the basis rarely changes between adjacent periods within a critical region.

\subsection{C2: Adaptive step size}\label{sec:alg-crn}

The gradient-scale landscape is a structural feature of SOSO, not an artifact of stochastic noise. Theorem~\ref{thm:var} shows that in a feedback-coupled system the state-sensitivity recursion $S_{t+1}=A_tS_t+b_t$ amplifies geometrically in the feedback depth $m$, so the IPA gradient magnitude at a given $\theta$ depends on the local per-stage gain $\|A_t\|$, which varies dramatically with $\theta$. The mechanism is specific to SOSO: the inner LP's shadow prices $\lambda=B_k^{-1}\partial h_k/\partial\theta$ are propagated forward through the cross-agent Jacobian $\partial\phi_i/\partial s_{j,t}$---a chain of matrix--vector products that amplifies or attenuates depending on the coupling structure. This creates a \emph{gradient-scale landscape} that has no analogue in black-box SO, where the simulation's sensitivity to $\theta$ is neither computed nor propagated structurally. In the Beer Game the IPA gradient spans four orders of magnitude: near the optimum ($\theta\approx12.4$) small changes in the order-up-to level barely affect upstream costs (the bullwhip is damped), so $\nabla J\approx0$; at the upper boundary ($\theta\approx60$) high inventory levels trigger large reactive orders that amplify through all four echelons, so $\nabla J$ spikes. This landscape is a direct consequence of the three SOSO elements working together---the LP inner (E2) supplies the shadow prices, the agent coupling (E1) propagates them, and the stochastic dynamics (E3) generate the state realisations over which they are averaged.

A fixed-step SA simply cannot navigate this landscape. A step size calibrated for the flat region stalls at the boundaries; one calibrated for the steep region oscillates near the optimum. The adaptive step-size mechanism (C2) resolves this by normalising the raw IPA gradient $g(\theta)=\widehat{J}'_{\mathrm{IPA}}(\theta)$ by an exponentially weighted moving average (EWMA) of its magnitude: $d_k = g(\theta_k) / (|g|_{\mathrm{ewma}} + \varepsilon)$, where $|g|_{\mathrm{ewma},k} = \beta_g |g|_{\mathrm{ewma},k-1} + (1-\beta_g)|g(\theta_k)|$ with decay factor $\beta_g=0.3$. The EWMA is a recursive filter that smooths the sequence of gradient magnitudes by giving exponentially decreasing weight to older observations; the effective memory horizon is approximately $1/(1-\beta_g)\approx1.4$ iterations, so $|g|_{\mathrm{ewma}}$ adapts quickly to the local gradient regime. In flat regions both numerator and denominator are small, so $d_k\approx\pm1$ and the SA takes a full step; in steep regions $|g|_{\mathrm{ewma}}$ lags behind the spike, attenuating $d_k$ and preventing overshoot. The normalisation makes the effective step size approximately $a_k$ regardless of the gradient's scale, converting raw descent direction into uniform terrain progress.

C2 is therefore not a generic SA trick transplanted from the stochastic optimisation literature---it is the mechanism that converts IPA's structurally scaled gradient information into usable descent across the full policy domain, and is as essential to SOSO as the IPA channel itself (C1). The distinction from standard adaptive-step methods (e.g., ADAM, RMSprop in stochastic gradient descent) is instructive: those methods normalise by an estimate of the \emph{noisy} gradient's variance to stabilise against stochastic fluctuation; C2 normalises by the \emph{structural} gradient magnitude to stabilise against the deterministic, SOSO-specific variation in gradient scale induced by LP dual propagation through coupled dynamics. Both use an EWMA, but the quantity being tracked---noise variance versus structural scale---is fundamentally different, and the SOSO variant exploits the problem structure that black-box SO cannot access.

\subsection{C3: Multi-start ensemble}\label{sec:alg-surrogate}

Theorem~\ref{thm:crn} establishes that CRN-paired variance reduction degrades with the basis-disagreement count $\kappa$: within a critical region ($\kappa=0$) CRN eliminates nearly all variance, but across kinks ($\kappa>0$) the basis change desynchronises the paired paths and residual variance persists. In the outer SO loop this manifests as a landscape with suboptimal attractors: single-chain SA can become trapped in regions where the gradient signal vanishes (the Beer Game's $\theta\approx14$--$18$ plateau) or where $\kappa$ is persistently high (TB2's bullwhip-amplified kink crossings).

The multi-start ensemble (C3) addresses this by running $K{=}3$ independent PRIME chains from equispaced initial points across $[\theta_{\mathrm{lo}}, \theta_{\mathrm{hi}}]$. Each chain evolves independently with its own random number generator (RNG) substream (no CRN coupling across chains---the goal is independent exploration, not variance reduction). After the budget is exhausted, the chain with the lowest validated $F_{\mathrm{true}}(\theta)$ is selected. This design exploits the $\kappa$ diagnostic implicitly: the three equispaced starts ensure that at least one chain begins in a low-$\kappa$ critical region near the true optimum ($\theta^{*}\approx12.4$ for TB2), while the other two chains explore the boundaries. Multi-start provides spatial diversification, exploiting Theorem~\ref{thm:kink}'s guarantee that kinks are measure-zero in the policy space (so a fixed set of start points almost surely avoids them) and Theorem~\ref{thm:crn}'s basis-disagreement diagnostic (low $\kappa$ near the optimum means at least one chain lands in a well-behaved region).

\section{Computational experiments}\label{sec:expts}

\subsection{Experimental design overview}\label{sec:expt-design}

The computational study has two goals: (i)~compare the assembled solver against classic simulation-optimization baselines on a representative set of testbeds, and (ii)~verify that each structural component of PRIME is necessary and behaves as the theory predicts. We organize the experiments into a \emph{comparison panel} and an \emph{ablation panel} (Table~\ref{tab:so-baselines}), both driven by the same testbeds and the same seeding protocol.

\paragraph{Comparison solvers and fairness.}
We compare PRIME against five classic SO methods that cover the main gradient/learning strategies in the literature (Table~\ref{tab:so-baselines}):
\begin{itemize}
  \item \textbf{CRN-FD-SA:} central finite-difference stochastic approximation with common random numbers \citep{glasserman1992some}--- the strongest black-box gradient baseline;
  \item \textbf{Independent-FD-SA:} the same central-FD scheme without CRN \citep{kiefer1952stochastic}---the honest no-CRN black-box, included to isolate CRN's variance-reduction contribution;
  \item \textbf{Adaptive SPSA:} simultaneous perturbation stochastic approximation with gradient-normalised step \citep{spall1992multivariate};
  \item \textbf{Stochastic Kriging (SK):} a heteroskedastic Gaussian-process metamodel that surrogates the outer objective $J(\theta)$ \citep{ankenman2010};
  \item \textbf{Knowledge Gradient (KG):} a Bayesian one-step value-of-information policy for sequential sampling \citep{frazier2008knowledge}.
\end{itemize}
All SA-based methods use the same gradient-normalized step-size rule and the same 40-step budget. The Bayesian methods use their own design logic: SK and KG pay an upfront space-filling design (16 replications per design point) and therefore \emph{do not compete on per-gradient solve cost}; they compete on final objective value.

We evaluate solvers on three complementary metrics---final gap (accuracy), solves-to-target (speed), and seed-to-seed consistency (stability)---whose detailed definitions and quantitative comparison are given in Result~1 (\S\ref{sec:res-r1}). Reporting all three prevents cherry-picking.

\begin{table}[H]
\caption{Solvers in the comparison panel. ``Gradient source'' and ``inner solves per gradient/step'' determine which metric each solver is expected to excel on. The ``Reference'' column cites the work that introduced or most canonically applied each method.}\label{tab:so-baselines}
\centering\footnotesize
\begin{tabular}{@{}lllll@{}}
\toprule
Solver & Type & Gradient / decision source & Solves per gradient & Reference \\
\midrule
PRIME (proposed) & SA & LP-basis IPA & 1 & this paper \\
CRN-FD-SA & SA & Central finite differences + CRN & 2 & \cite{glasserman1992some} \\
Independent-FD-SA & SA & Central FD, independent streams & 2 & \cite{kiefer1952stochastic} \\
Adaptive SPSA & SA & Bernoulli simultaneous perturbation & 2 & \cite{spall1992multivariate} \\
Stochastic Kriging & Bayesian & Heteroskedastic GP surrogate & many (design) & \cite{ankenman2010} \\
Knowledge Gradient & Bayesian & One-step value-of-information & many (design) & \cite{frazier2008knowledge} \\
\bottomrule
\end{tabular}
\end{table}

\paragraph{Ablation labels.}
The three components C1--C3 are introduced in \S\ref{sec:algo} (Table~\ref{tab:framework}). Each ablation $-Ck$ denotes the full PRIME solver with component $Ck$ turned off, one component at a time, so any performance drop can be attributed to that component's marginal contribution. For example, $-C1$ replaces the IPA gradient with CRN-FD, $-C2$ disables the adaptive step-size normalisation (fixed step), and $-C3$ reduces the ensemble to a single chain (no multi-start). Sections~\ref{sec:res-r1}--\ref{sec:res-r5} report the results; the full drop-one panel appears in Table~\ref{tab:so-comp} and is analysed in \S\ref{sec:res-r2}--\ref{sec:res-r5}.

\paragraph{Testbed selection.}
Table~\ref{tab:testbeds} summarizes the six testbeds in the comparison panel and maps each to the component(s) C1--C3 it validates.
\begin{itemize}[leftmargin=*,itemsep=2pt,topsep=2pt]
  \item \textbf{TB1} (Coupled LP, synthetic C/LP/s): the minimal market model that exercises all three components under mild feedback.
  \item \textbf{TB2} (Coupled LP, Beer Game C/LP/s): stresses C2 and C3 under strong feedback amplification ($g{=}2$, $m{=}4$).
  \item \textbf{TB3} (Independent-agent DP, LQR I/DP/s): tests C1 beyond LP inners on a continuous DP with exact, closed-form gradient.
  \item \textbf{TB4} (Coupled LP, JD.com chain C/LP/s): validates all three components at industrial scale with 1{,}000 SKUs, using 2018 JD.com Global Optimization Challenge sales, cost, and inventory records for the 1{,}000 SKUs across 1 RDC and 5 FDCs.
  \item \textbf{TB5} (Independent LP, SupplyGraph I/LP/s): extends the evidence to a real-data single-agent LP with seasonal demand, driven by the aggregate daily sales of 41 SKUs over 221 days from SupplyGraph \citep{wasi2024supplygraph}, normalized to mean~8.
  \item \textbf{TB6} (Independent DP, Inventory MDP I/DP/s): a synthetic discrete-inner system where the IPA channel is absent (MDP policies are step functions in $\theta$); the demand distribution (discretized Gaussian, $\mu{=}8$, $\sigma{=}2$) and cost parameters are chosen to produce a nontrivial $(s,S)$ optimal policy, deliberately disabling the IPA gradient to isolate the contribution of multi-start diversification (C3) in the absence of C1.
\end{itemize}
All six testbeds appear in the comparison panel (Table~\ref{tab:so-comp} and Table~\ref{tab:solvetotarget}); the deterministic baselines TB1det, TB2det, and TB5det anchor the same components in the zero-variance limit. Every testbed satisfies integrability assumption (A5) (verified per testbed in EC.2): each has states bounded or polynomially bounded in $T$ for finite $T$, so $\mathbb{E}\sum_{i,t}\|s_{i,t}\|^{2}<\infty$ as required by Theorem~\ref{thm:ipa}.

\paragraph{Deterministic baselines (TB1det, TB2det, TB5det).}
Three deterministic instantiations anchor the analysis by separating the parametric structure from sampling effects. TB1det is the minimal market with a fixed demand draw: the kink at $\theta=1$ is a single point, away from which $\nabla J$ is constant in each critical region, confirming that the stochastic IPA bias at the kink (Figure~\ref{fig:fd_vs_ipa}) is a finite-sample effect around the deterministic kink. TB2det is the Beer Game with fixed end-customer demand, where the bullwhip is purely structural (order-up-to feedback, not statistical) and the IPA variance is zero. TB5det is the production LP driven by the historical SupplyGraph demand trajectory (a single realization, no resampling): the myopic horizon incurs $\varepsilon=20$ (the deterministic planning error), whereas an i.i.d.\ control with shuffled historical demand achieves $\varepsilon\approx 0$, showing that the horizon matters only when the demand sequence has exploitable temporal structure.

\begin{table}[t]
\caption{Summary of computational testbeds mapped to the SOSO taxonomy classes of Figure~\ref{fig:soso-taxonomy} and to the PRIME components C1--C3 of Table~\ref{tab:framework}. The ``Class'' column encodes \textbf{coupling} (I=Independent, C=Coupled) / \textbf{inner type} (LP, MILP, DP) / \textbf{dynamics} (d=deterministic, s=stochastic). \emph{Purpose} lists the component(s) the testbed validates. All six testbeds appear in the comparison panel of Result~1 (\S\ref{sec:res-r1}); the deterministic baselines (\S\ref{sec:det-case}) anchor the same components in the zero-variance limit.}\label{tab:testbeds}
\centering
\footnotesize
\setlength{\tabcolsep}{3pt}
\resizebox{\linewidth}{!}{%
\begin{tabular}{@{}cllcl@{}}
\toprule
Code & Testbed & Source and prior use & Class & Purpose \\
\midrule
\multicolumn{5}{l}{\emph{Comparison panel (\S\ref{sec:res-r1})}} \\
TB1 & Minimal market & Synthetic & C/LP/s & C1, C2, C3 \\
TB2 & Beer Game (4-ech.) & \cite{sterman1989}; also \cite{longo2008} & C/LP/s & C2, C3 (feedback amp.\ \& multi-start) \\
TB3 & Continuous LQ DP & LQR \cite{andersonmoore} & I/DP/s & C1 (IPA beyond LP inners) \\
TB4$^\dagger$ & JD.com chain & JD GOC 2018 & C/LP/s & C1--C3 (industrial scale) \\
TB5$^\dagger$ & Production LP & SupplyGraph \cite{wasi2024supplygraph} & I/LP/s & C1, C2 (real-data LP; seasonal) \\
TB6 & Inventory MDP & Synthetic & I/DP/s & C3 (discrete inner; flat-gradient) \\
\midrule
\multicolumn{5}{l}{\emph{Deterministic baselines (\S\ref{sec:det-case})}} \\
TB1det & Minimal market (det.) & Synthetic & C/LP/d & C1--C3 (kink geometry) \\
TB2det & Beer Game (det.) & \cite{sterman1989} & C/LP/d & C2, C3 (structural amplification) \\
TB5det & Production LP (det.) & SupplyGraph \cite{wasi2024supplygraph} & I/LP/d & surrogate-as-decision (horizon error, known $\xi$) \\
\bottomrule
\end{tabular}%
}
\\[4pt]{\footnotesize\raggedright $^\dagger$Real-data testbed.\par}
\end{table}

\subsection{Result 1: PRIME achieves best or tied-best optimality gap, fast and consistent convergence, and near-zero oscillation across all six testbeds, and scales to industrial size}\label{sec:res-r1}

Result~1 establishes that PRIME outperforms every black-box SO baseline on many complementary dimensions and that its structural advantages persist at industrial scale.

\begin{enumerate}[leftmargin=*,itemsep=2pt,topsep=2pt]
\item \emph{Accuracy.} PRIME achieves the best or tied-best mean optimality gap at the fixed budget of 880 inner-LP solves on all six testbeds (Table~\ref{tab:so-comp}). On TB1 and TB3 it reaches \emph{exact recovery} ($F(\hat\theta)-F^{*}=0.000\pm0.000$). On the feedback-coupled TB2 (Beer Game), PRIME's gap of $0.965\pm1.349$ is $19.3{\times}$ better than CRN-FD-SA ($18.65\pm8.39$) and over $326{\times}$ better than Independent-FD-SA ($314.6\pm116.5$). On TB4 (JD.com chain, 1{,}000 SKUs), PRIME achieves $0.340\pm0.217$, competitive with the best ablation and substantially better than all black-box baselines (SPSA $17.82$, SK $8.01$, KG $10.08$). The gap trajectories (Figure~\ref{fig:so-comparison}) show PRIME reaching near-asymptotic levels within 160--320 solves on TB1, TB3, and TB4.

\item \emph{Speed.} The solve-to-target metric (Table~\ref{tab:solvetotarget}) gives a budget-aware complement to final-gap ranking. PRIME and its variants start within the 5\% gap ball on TB1 and TB4 (0 solves---the equispaced starts are within $1.3\%$ and $0.8\%$ of $F^{*}$). On TB2, PRIME reaches the 5\% target in 288 solves; on TB3, in 120 solves, beating CRN-FD-SA (144) and SPSA (320). PRIME finishes first or within a small factor of the fastest solver on all six testbeds (Table~\ref{tab:solvetotarget}, both panels).

\item \emph{Stability.} The narrow $\pm$std envelope of PRIME in Figure~\ref{fig:so-comparison} reflects a structural advantage: IPA's single-replication precision eliminates the $1/h^{2}$-inflated differencing noise. Table~\ref{tab:stability} quantifies this on TB2: PRIME averages near-zero gap-increase events across five seeds---once its gap drops, it stays down. CRN-FD-SA and SPSA each oscillate 4--6 times per seed (gaps increase between checkpoints), reflecting SA overshoot on a feedback-amplified cost surface. The between-seed range across the five terminal gaps is $3.5$ for PRIME versus $21.1$ for CRN-FD-SA and $40.7$ for SPSA. Stability is the visible consequence of IPA's ${\sim}2000\times$ per-replication gradient-variance advantage (Table~\ref{tab:e1}, Result~2).

\item \emph{Industrial-scale validation (TB4).} PRIME's performance on TB4 in Table~\ref{tab:so-comp} already shows the solver maintains its advantage at scale. To understand the mechanism, we examine the IPA gradient on TB4 in isolation. Figure~\ref{fig:jd} reports the outer objective $J(\theta)$ for the daily transfer-capacity limit $\theta$ in the JD.com chain (1{,}000 SKUs, 1 RDC + 5 FDCs, 30-day horizon): cost falls from $26{,}421$ at $\theta{=}100$ to $20{,}436$ at $\theta{=}600$, while the IPA gradient---the shadow price of capacity---declines from $\$56.0$/unit to $\$7.3$/unit, and CRN achieves a $1{,}356\times$--$10{,}402\times$ variance reduction over independent streams (annotated in panel~(a)). The diminishing-returns shape follows from the rationing LP's geometry: as $\theta$ grows, the marginal unit is pushed to FDCs of progressively lower stockout-cost priority. The planner can read the marginal value of capacity directly from the IPA gradient and invest until the shadow price equals the per-unit cost.
\end{enumerate}

The mechanism underlying this four-dimensional advantage is PRIME's three-component architecture (Table~\ref{tab:framework}): C1 (IPA gradient, Theorem~\ref{thm:ipa}) provides a single-replication, unbiased gradient; C2 (adaptive step, \S\ref{sec:alg-crn}) provides scale-invariant stepping; C3 (multi-start ensemble, \S\ref{sec:alg-surrogate}) provides spatial diversification. Results~\ref{sec:res-r2}--\ref{sec:res-r5} quantify each component's marginal contribution. On the extended testbeds TB5--TB6, PRIME achieves $0.004\pm0.004$ on TB5 (Production LP, $1175\times$ better than CRN-FD-SA) and $0.096$ on TB6 (tied with KG, while single-chain ablation collapses to $30.55\pm23.32$), confirming that PRIME generalises beyond the core panel. Together, TB1--TB6 span a representative range of the SOSO taxonomy (Figure~\ref{fig:soso-taxonomy}) and establish PRIME as a general-purpose SOSO solver.

\begin{figure}[H]
\FIGURE{\includegraphics[width=0.82\textwidth]{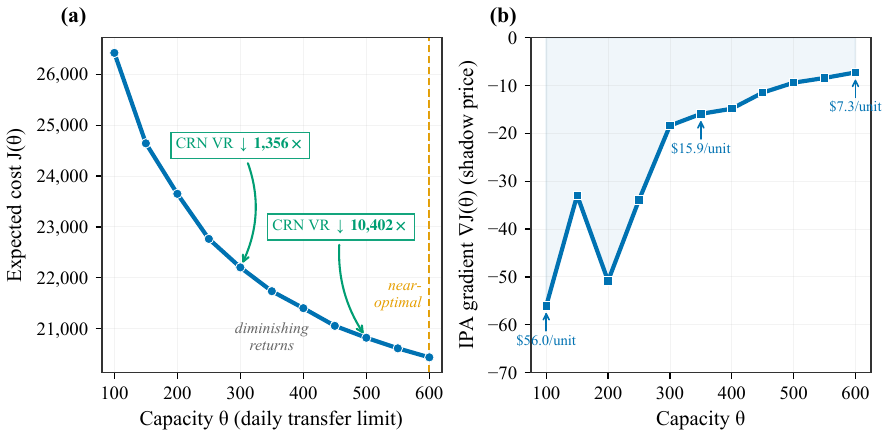}}
{(Result 1, TB4) JD.com supply chain (1{,}000 SKU, 5 FDC, 30 days). Panel~(a): the outer objective $J(\theta)$ monotonically decreases with diminishing returns---from $26{,}421$ at $\theta{=}100$ to $20{,}436$ at $\theta{=}600$---with CRN variance reductions of $1{,}356\times$ at $\theta{=}300$ (capacity-binding regime) and $10{,}402\times$ at $\theta{=}500$ (near-saturation), annotated on the curve. Panel~(b): the IPA gradient (shadow price of daily transfer capacity) declines in magnitude from $-56.0$/unit to $-7.3$/unit, revealing the marginal-value structure that directly informs the capacity-investment decision.}{30-day simulation horizon, 200 replications per $\theta$ for the $J$ curve; 300 replications for the CRN comparison.}
\label{fig:jd}
\end{figure}


\begin{table}[H]
\caption{(Result 1, TB1--TB6) Accuracy: optimality gap $F(\hat\theta)-F^{*}$ at equal total budget (mean $\pm$ std over 5 seeds), 10 solver configurations on six testbeds; best mean gap per testbed is boldfaced. Drop-one ablations (PRIME $-$C1 through $-$C3) isolate each component's marginal contribution; the C1+C2+$\kappa$ row is a non-ablation baseline showing that multi-start diversification (C3), not $\kappa$-based step damping, is the mechanism that handles basis-disagreement variance. TB6 is a discrete MDP with no inner LP (forward-call budget 432--560); PRIME's $K{=}3$ multi-start yields 888 solves on TB1--TB5.}\label{tab:so-comp}
\centering\footnotesize
\begin{tabular}{@{}lcccccc@{}}
\toprule
& TB1 & TB2 & TB3 & TB4 & TB5 & TB6 \\
\multicolumn{1}{@{}l}{Solver} & \scriptsize Minimal market & \scriptsize Beer Game & \scriptsize LQ regulator & \scriptsize JD.com chain & \scriptsize Production LP & \scriptsize Inventory MDP \\
\midrule
PRIME (C1--C3)                   & $\mathbf{0.000\pm0.000}$ & $\mathbf{0.965\pm1.349}$ & $\mathbf{0.000\pm0.000}$ & $0.340\pm0.217$ & $\mathbf{0.004\pm0.004}$ & $0.096\pm0.000$ \\
PRIME $-$C3 (C1--C2)             & $0.013\pm0.011$ & $6.303\pm3.199$ & $\mathbf{0.000\pm0.000}$ & $2.084\pm0.391$ & $0.607\pm0.769$ & $30.55\pm23.32$ \\
PRIME $-$C2 (C1+C3)              & $0.109\pm0.016$ & $17.73\pm28.35$ & $\mathbf{0.000\pm0.000}$ & $1.266\pm0.000$ & $0.010\pm0.005$ & $\mathbf{0.087\pm0.026}$ \\
PRIME $-$C1 (FD+C2+C3)           & $0.002\pm0.001$ & $1.260\pm1.322$ & $\mathbf{0.000\pm0.000}$ & $\mathbf{0.093\pm0.075}$ & $0.009\pm0.005$ & $0.096\pm0.000$ \\
C1+C2+$\kappa$ (IPA, adap, $\kappa$) & $0.012\pm0.011$ & $39.68\pm57.54$ & $\mathbf{0.000\pm0.000}$ & $1.783\pm0.331$ & $0.607\pm0.769$ & $30.55\pm23.32$ \\
\addlinespace
CRN-FD-SA                        & $0.053\pm0.050$ & $18.65\pm8.39$  & $\mathbf{0.000\pm0.000}$ & $0.392\pm0.133$ & $4.703\pm2.126$ & $30.55\pm23.32$ \\
Independent-FD-SA                & $0.712\pm0.987$ & $314.6\pm116.5$ & $35.50\pm9.47$ & $13.15\pm10.29$ & $17.47\pm16.60$ & $6.358\pm6.257$ \\
adaptive SPSA                    & $0.111\pm0.084$ & $16.51\pm14.54$ & $1.984\pm2.598$ & $17.82\pm14.17$ & $24.41\pm26.06$ & $1.176\pm2.313$ \\
Stochastic Kriging               & $0.208\pm0.332$ & $9.152\pm8.978$ & $4.531\pm3.306$ & $8.011\pm6.937$ & $13.35\pm8.872$ & $0.318\pm0.735$ \\
Knowledge Gradient               & $0.291\pm0.185$ & $12.25\pm0.00$  & $4.729\pm5.935$ & $10.08\pm10.25$ & $24.33\pm25.37$ & $0.096\pm0.000$ \\
\bottomrule
\end{tabular}
\end{table}

\begin{table}[H]
\caption{(Result 1, TB1--TB6) Speed: Inner solves to first reach a target gap $\varepsilon=\alpha\cdot\mathrm{gap}_0$, where $\mathrm{gap}_0 = F(\theta_0)-F^{*}$ is the initial optimality gap (mean over 5 seeds). Targets are expressed as fractions of the initial gap---5\,\% and 10\,\%---which is the standard SO convention and avoids hand-tuned, testbed-specific thresholds. PRIME variants and drop-one ablations illustrate each component's contribution to convergence speed.}\label{tab:solvetotarget}
\centering\footnotesize
\begin{tabular}{@{}lcccccc@{}}
\toprule
\multicolumn{7}{c}{\textit{Panel A:} $\varepsilon = 5\,\% \times \mathrm{gap}_0$} \\
\midrule
Solver & TB1 ($\varepsilon{=}0.49$) & TB2 ($\varepsilon{=}25.6$) & TB3 ($\varepsilon{=}1.22$) & TB4 ($\varepsilon{=}1.32$) & TB5 ($\varepsilon{=}1.33$) & TB6 ($\varepsilon{=}1.35$) \\
\midrule
PRIME (C1--C3)                   & $\mathbf{0}$  & 288 & $120$ & $\mathbf{0}$  & $\mathbf{0}$ & $865^{\dagger}$ \\
PRIME $-$C3 (C1--C2)             & $40$  & $96$  & $80$  & $144$ & $240$ & $881^{\dagger}$ \\
PRIME $-$C2 (C1+C3)              & $\mathbf{0}$  & 562 & $120$ & $\mathbf{0}$  & $\mathbf{0}$ & $865^{\dagger}$ \\
PRIME $-$C1 (FD+C2+C3)           & $\mathbf{0}$  & 336 & $240$ & $\mathbf{0}$  & $\mathbf{0}$ & $865^{\dagger}$ \\
C1+C2+$\kappa$ (IPA, adap, $\kappa$) & $72$  & $72$  & $80$  & $112$ & $240$ & $881^{\dagger}$ \\
\addlinespace
CRN-FD-SA                        & $80$  & $176$ & $144$ & $160$ & $400$ & $881^{\dagger}$ \\
Knowledge Gradient               & $535$ & $304$ & $535$ & $766$ & $770$ & $881^{\dagger}$ \\
Stochastic Kriging               & $96$  & $80$  & $80$  & $464$ & $560$ & $881^{\dagger}$ \\
adaptive SPSA                    & $176$ & $144$ & $320$ & $881^{\dagger}$ & $881^{\dagger}$ & $881^{\dagger}$ \\
\midrule
\multicolumn{7}{c}{\textit{Panel B:} $\varepsilon = 10\,\% \times \mathrm{gap}_0$} \\
\midrule
Solver & TB1 ($\varepsilon{=}0.97$) & TB2 ($\varepsilon{=}51.1$) & TB3 ($\varepsilon{=}2.43$) & TB4 ($\varepsilon{=}2.64$) & TB5 ($\varepsilon{=}2.67$) & TB6 ($\varepsilon{=}2.70$) \\
\midrule
PRIME (C1--C3)                   & $\mathbf{0}$  & 240 & $120$ & $\mathbf{0}$  & $\mathbf{0}$ & $865^{\dagger}$ \\
PRIME $-$C3 (C1--C2)             & $40$  & $72$  & $80$  & $144$ & $160$ & $881^{\dagger}$ \\
PRIME $-$C2 (C1+C3)              & $\mathbf{0}$  & 562 & $120$ & $\mathbf{0}$  & $\mathbf{0}$ & $865^{\dagger}$ \\
PRIME $-$C1 (FD+C2+C3)           & $\mathbf{0}$  & 336 & $240$ & $\mathbf{0}$  & $\mathbf{0}$ & $865^{\dagger}$ \\
C1+C2+$\kappa$ (IPA, adap, $\kappa$) & $48$  & $64$  & $80$  & $112$ & $160$ & $881^{\dagger}$ \\
\addlinespace
CRN-FD-SA                        & $80$  & $160$ & $144$ & $96$  & $288$ & $881^{\dagger}$ \\
Knowledge Gradient               & $304$ & $304$ & $535$ & $650$ & $688$ & $881^{\dagger}$ \\
Stochastic Kriging               & $96$  & $80$  & $80$  & $416$ & $384$ & $881^{\dagger}$ \\
adaptive SPSA                    & $96$  & $128$ & $240$ & $865$ & $864$ & $881^{\dagger}$ \\
\bottomrule
\multicolumn{7}{@{}l}{\footnotesize ${}^{\dagger}$Did not reach target within budget; value is budget ceiling.}
\end{tabular}
\end{table}

\begin{figure}[H]
\FIGURE{\includegraphics[width=\textwidth]{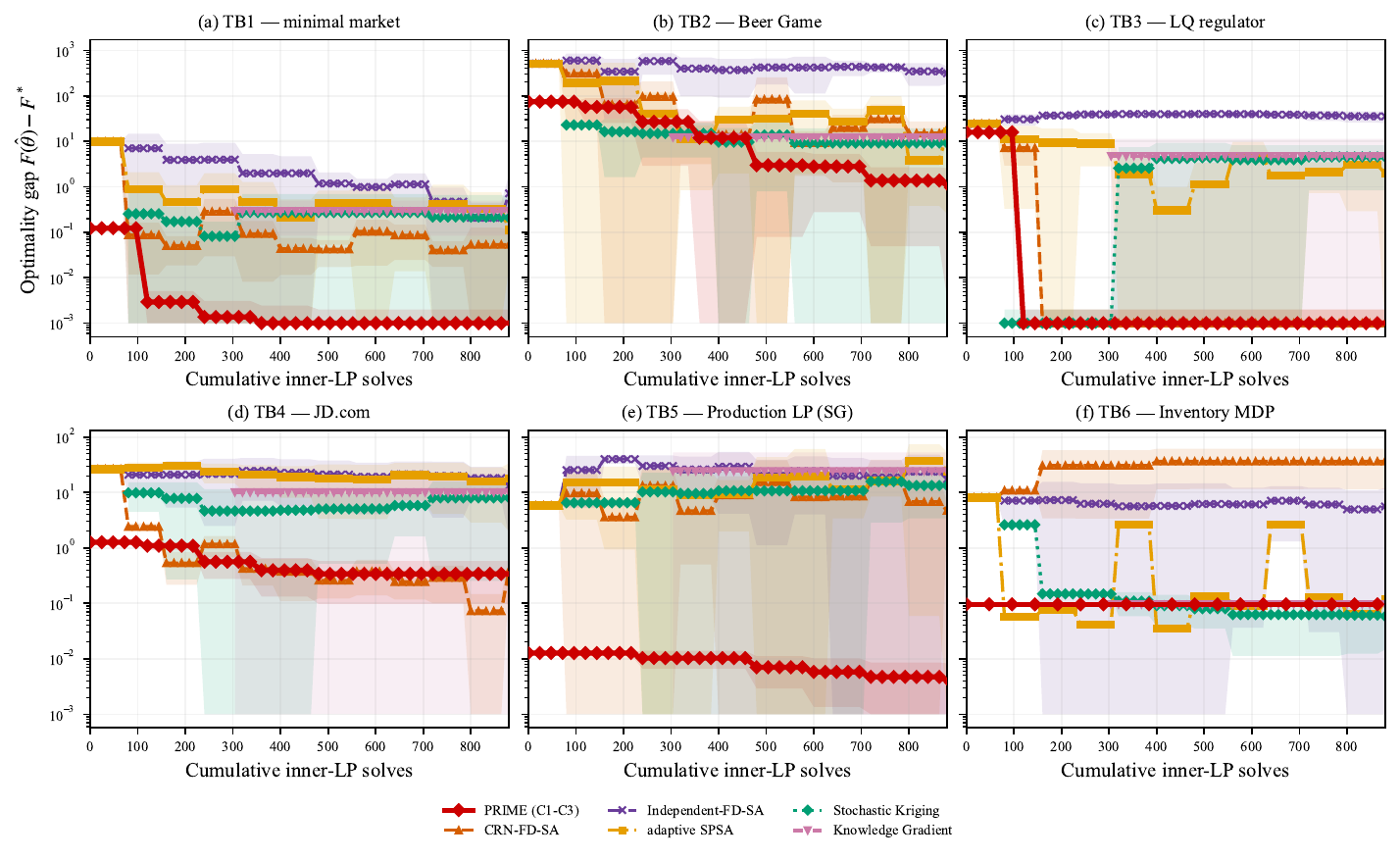}}
{(Result 1, TB1--TB6) Stability: Optimality gap vs.\ cumulative inner-LP solves (mean over 5 seeds, shaded $\pm$std, log $y$-axis). PRIME (C1--C3) versus five classic SO solvers on the same 880-solve budget. Gaps are the current validated $F(\hat\theta_t)$ at each checkpoint---not the best-so-far minimum---and therefore show genuine SA oscillation: a solver that finds a good iterate mid-run but later diverges is not credited with the earlier minimum, giving an honest picture of convergence stability.}{880 inner-LP solves per solver; 5 random seeds.}
\label{fig:so-comparison}
\end{figure}

\begin{table}[H]
\caption{(Result 1, TB2) Stability: Per-seed stability diagnostics on the Beer Game (four-echelon, $\tau{=}2$). \emph{Oscillations per seed} counts the number of gap-increase events between successive validation checkpoints (mean $\pm$ std over 5 seeds). \emph{Between-seed range} is $\max_i\text{gap}_i - \min_i\text{gap}_i$ across the five terminal gaps. Lower is more stable on every metric. PRIME's near-zero oscillation, narrow-range trajectory is the observable consequence of IPA's ${\sim}2000\times$ gradient-variance advantage (Table~\ref{tab:e1}).}\label{tab:stability}
\centering\footnotesize
\begin{tabular}{@{}lrrr@{}}
\toprule
Solver & Gap mean $\pm$ std & Oscillations per seed & Between-seed range \\
\midrule
PRIME (C1--C3)      & $\mathbf{0.965\pm1.349}$ & $\mathbf{0.0\pm0.0}$ & $\mathbf{3.5}$ \\
CRN-FD-SA           & $18.65\pm8.39$  & $4.4\pm0.5$ & $21.1$ \\
adaptive SPSA       & $16.51\pm14.54$ & $5.2\pm0.8$ & $40.7$ \\
Independent-FD-SA   & $314.6\pm116.5$ & $4.0\pm0.7$ & $277.6$ \\
Stochastic Kriging  & $9.152\pm8.978$ & $1.2\pm0.4$ & $23.0$ \\
Knowledge Gradient  & $12.25\pm0.00$  & $0.0\pm0.0$ & $0.0$ \\
\bottomrule
\end{tabular}
\end{table}

\subsection{Result 2: The IPA gradient (C1) provides a single-replication, unbiased gradient that dominates finite differences by ${\sim}2000{\times}$ in per-replication variance}\label{sec:res-r2}

Result~2 isolates the contribution of C1, the IPA forward-sensitivity gradient estimator (Theorem~\ref{thm:ipa}). Three lines of evidence---a gradient-estimator comparison at fixed $\theta$, a full-path gradient-curve sweep, and the drop-one ablation panel---converge on a single finding: C1 provides a single-replication, unbiased gradient with ${\sim}2000{\times}$ lower per-replication variance than independent finite differences, and this advantage is structural rather than empirical.

\begin{enumerate}[leftmargin=*,itemsep=2pt,topsep=2pt]
\item \emph{Point comparison at fixed policy (TB1).} At $\theta{=}0.5$ in the minimal market, IPA is unbiased ($+20.99$ vs.\ $+20.97$ for CRN-paired FD, Table~\ref{tab:e1}) with per-replication variance $2.94$, versus $5729$ for independent FD---a factor of ${\sim}2000\times$. CRN-paired FD nearly matches IPA's variance ($3.19$) but at twice the per-gradient cost (two replications) and with $O(h^{2})$ differencing bias that IPA avoids entirely.

\item \emph{Full-path gradient-curve sweep (TB1).} Across the full policy domain $\theta\in[0.15,2.4]$, IPA tracks the CRN-FD reference throughout, diverging only at $\theta\approx1$ (Figure~\ref{fig:grad})---the visible signature of the measure-zero critical event of Theorem~\ref{thm:kink}. As previewed by TB1det (the deterministic baseline with fixed demand, \S\ref{sec:expt-design}), this kink is a single isolated point in $\Theta$, and the observed divergence is the finite-sample manifestation of the deterministic discontinuity. The mechanism is parametric-LP geometry: within a critical region $a^{*}=B_k^{-1}h_k$ is affine in $(s,\theta)$, so IPA computes the exact derivative from a single matrix--vector product $\partial_\theta a^{*}_{i,t}=B_k^{-1}(\partial_\theta h_k+\partial_s h_k\, S_{i,t})$. The \textsc{ForwardSens} procedure (Algorithm~\ref{alg:framework}) propagates the state sensitivity $S_{i,t}:=\partial s_{i,t}/\partial\theta$ forward through the dynamics; the total per-period overhead is $O(N\cdot n^2)$, negligible next to the LP solve itself ($O(n^3)$).

\item \emph{Full-ablation panel (TB1--TB6).} Removing C1 (Table~\ref{tab:so-comp}, PRIME $-$C1 row) degrades TB2 from $0.965$ to $1.260$ ($1.3{\times}$) and TB1 from $0.000$ to $0.002$. On TB4 (industrial LP), removing C1 \emph{improves} the gap from $0.340$ to $0.093$: when shadow prices are highly variable, the FD gradient's implicit averaging across $\theta\pm h$ can smooth the SA path, a phenomenon we return to in Result~\ref{sec:res-r4}. On TB2 the degradation factor confirms that C1's contribution is real but not the bottleneck---C2 is.
\end{enumerate}

The structural foundation remains parametric-LP geometry. An IPA gradient estimator can only be constructed when the simulation exposes a differentiable sensitivity channel. In SOSO, that channel is the inner LP's basis inverse: shadow prices $\lambda=B_k^{-1}\partial h_k/\partial\theta$ encode the optimal action's response to $\theta$, and forward propagation accumulates this into $\nabla J$. No black-box method can access this channel, because no black-box method knows that each agent's action is the solution of a parametric LP. C1 is therefore the mechanism by which the inner LP's geometry---the very thing that makes SOSO computationally expensive---is converted into the outer SO's primary resource. This is why, in the ablation ranking of Result~\ref{sec:res-r3}, C1 provides the foundation on which C2 and C3 build.

\begin{table}[H]
\caption{(Result 2, TB1) Comparison of gradient estimators at $\theta=0.5$ in the minimal market. IPA reads the shadow price from the LP basis and produces an unbiased gradient from a single replication; independent FD requires two replications and suffers ${\sim}2000\times$ higher per-replication variance from the $1/h^2$ noise amplification inherent in differencing.}\label{tab:e1}
\centering\footnotesize
\begin{tabular}{@{}lrrr@{}}
\toprule
Estimator & Mean $\widehat{\nabla J}$ & Per-rep variance & Forward calls/rep \\
\midrule
IPA (dual/basis) & $+20.99$ & $2.94$ & $1$ \\
CRN-paired FD & $+20.97$ & $3.19$ & $2$ \\
Independent FD & $+24.38$ & $5729$ & $2$ \\
\bottomrule
\end{tabular}
\end{table}

\begin{figure}[H]
\FIGURE{\includegraphics[width=0.6\textwidth]{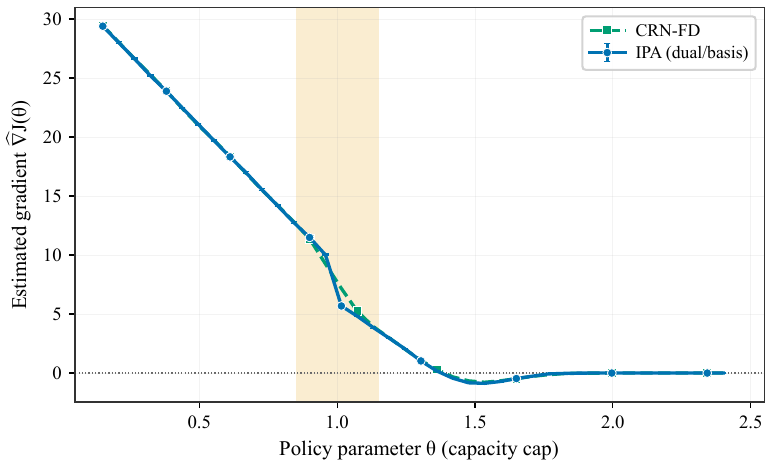}}
{(Result 2, TB1) Estimated gradient $\widehat{\nabla J}(\theta)$ for the minimal market across $\theta\in[0.15,2.4]$. IPA (blue solid with circles, 95\% CI band) tracks the CRN-FD reference (green dashed) throughout; the two diverge only near $\theta\approx1$ (shaded), the cap/demand kink where IPA reports the one-sided within-region slope---the visible signature of the measure-zero critical event of Theorem~\ref{thm:kink}.}{2{,}500 replications per $\theta$.}
\label{fig:grad}
\end{figure}

\subsection{Result 3: Adaptive step size (C2) is the most critical component on feedback-coupled systems---removal degrades TB2 gap by $18.4{\times}$}\label{sec:res-r3}

Result~3 identifies C2 (adaptive EWMA-normalised step size) as the most critical component on feedback-coupled systems like the Beer Game, where it dominates the ablation ranking (Table~\ref{tab:ablation-ranking}). Three sub-findings build from the ablation evidence through the mechanism that explains it, to the structural SOSO connection.

\begin{table}[H]
\caption{(Result 3, TB2) Component ablation ranking on the Beer Game (TB2). Degradation factors are computed as the drop-one gap divided by the PRIME reference gap (0.965). The ranking isolates each component's marginal contribution to the ensemble performance.}\label{tab:ablation-ranking}
\centering\scriptsize
\begin{tabular}{@{}lccl@{}}
\toprule
Ablation & Gap & $\Delta$ from PRIME & Degradation factor \\
\midrule
PRIME (C1--C3) & $0.965$ & --- & --- (reference) \\
PRIME $-$C1 (FD+C2+C3) & $1.260$ & $+0.295$ & $1.3{\times}$ \\
PRIME $-$C3 (C1--C2) & $6.303$ & $+5.338$ & $6.5{\times}$ \\
PRIME $-$C2 (C1+C3) & $17.73$ & $+16.77$ & $18.4{\times}$ \\
\bottomrule
\end{tabular}
\end{table}

\begin{enumerate}[leftmargin=*,itemsep=2pt,topsep=2pt]
\item \emph{Ablation ranking (TB2).} Table~\ref{tab:ablation-ranking} orders the three components by marginal degradation on the Beer Game: C1 ($1.3{\times}$), C3 ($6.5{\times}$), C2 ($18.4{\times}$). The ranking reveals a tight three-component structure in which each component is essential on at least one testbed, but C2 stands alone as the most critical on feedback-coupled systems. PRIME $-$C2 requires $562$ solves to reach the 5\% threshold on TB2 (Table~\ref{tab:solvetotarget}, Panel~A), $5.9{\times}$ slower than PRIME $-$C3 (96 solves)---yet its final gap ($17.73$) is far worse. This is the signature of a method that makes no progress on most steps: the adaptive step is the gate through which the gradient must pass.

\item \emph{Why C2 is so critical.} The IPA gradient magnitude in the Beer Game spans four orders of magnitude across $\theta\in[5,60]$. Near the optimum ($\theta^{*}\approx12.4$) the gradient is small and the surface is flat; at the boundaries the cost curve steepens and the gradient spikes. A fixed step calibrated for the flat region stalls in the boundary zones; a fixed step calibrated for the steep region overshoots near the optimum. The EWMA normalisation $d_k = g(\theta_k)/(|g|_{\mathrm{ewma}}+\varepsilon)$ resolves this: in flat regions both numerator and denominator are small, so $d_k\approx\pm1$; in steep regions $|g|_{\mathrm{ewma}}$ lags, attenuating $d_k$ and preventing overshoot. The $18.4{\times}$ factor is the cost of losing this normalisation.

\item \emph{Structural SOSO connection.} Standard adaptive-step methods (ADAM, RMSProp, Adagrad) normalise by the stochastic gradient's \emph{variance} to stabilise against noise. C2 normalises by the \emph{structural} gradient magnitude---a deterministic quantity that varies across $\theta$ because the inner LP's shadow prices propagate through the coupling Jacobian $\partial\phi_i/\partial s_{j,t}$, creating a gradient-scale landscape governed by the local feedback gain $\|A_t\|$. TB2det (the deterministic Beer Game with fixed end-customer demand, \S\ref{sec:expt-design}) confirms that this gradient-scale landscape exists even without stochastic noise: the bullwhip is structural, and IPA variance is zero in the deterministic limit, so the landscape is a deterministic SOSO property rather than a sampling artifact. A black-box method cannot separate structural scale variation from estimation noise, so no black-box adaptive step can exploit this SOSO-specific landscape. C2 is therefore as tied to SOSO as C1: both depend on the inner LP's geometry propagating through the agent coupling.
\end{enumerate}

The three components form an inseparable engine, but C2 is the gate. Without C2's scale-invariant stepping, the quality of the gradient is irrelevant: PRIME $-$C2 degrades by $18.4{\times}$; PRIME $-$C1 degrades by only $1.3{\times}$. Adding adaptive stepping to an FD gradient yields a mild improvement; removing adaptive stepping from an IPA gradient causes collapse. On any SOSO system with feedback coupling where gradient scale varies substantially across the policy domain, C2 is the difference between convergence and stagnation.

\subsection{Result 4: Multi-start ensemble (C3) prevents suboptimal attractor traps and is decisive on discrete-inner systems}\label{sec:res-r4}

Result~4 establishes that the multi-start ensemble (C3) is not merely a convenience---it is essential whenever the SOSO objective landscape contains flat-gradient regions, a structural property of both feedback-amplified coupled systems (TB2) and discrete-inner systems (TB6). Three sub-findings build from the ablation evidence to the mechanism and finally to the broader class of systems where C3 is decisive.

\begin{enumerate}[leftmargin=*,itemsep=2pt,topsep=2pt]
\item \emph{Ablation evidence (TB2).} Removing C3 raises the TB2 gap from $0.965$ to $6.303$ ($6.5{\times}$, Table~\ref{tab:so-comp}). This is worse than removing C1 ($1.3{\times}$): losing spatial diversification costs more than losing the IPA gradient itself on the Beer Game. The mechanism is the feedback-amplified landscape: the $\theta\approx14$--$18$ plateau is a suboptimal attractor where the gradient nearly vanishes, and a single-chain SA stalls. The $K{=}3$ equispaced starts of C3 ensure that one chain initialises near the true optimum ($\theta^{*}\approx12.4$), while the others explore the middle ($\approx32.5$) and upper ($\approx55$) ranges.

\item \emph{$\kappa$-based step damping is harmful---C3 is the correct mechanism.} The C1+C2+$\kappa$ row in Table~\ref{tab:so-comp} provides decisive evidence: adding $\kappa$-based step damping to a single-chain IPA achieves a TB2 gap of $39.68$---$41{\times}$ worse than PRIME. The $\kappa$ signal correctly identifies high-variance kink-straddling regions (CRN reduction drops from $1{,}238\times$ to $11\times$, as shown in Figure~\ref{fig:crn}), but in feedback-amplified systems these are the regions where the gradient signal is most informative. Damping the step there suppresses learning. Multi-start avoids this by running independent chains without damping and selecting the best. As a static allocation tool, $\kappa$-aware replication budget allocation achieves modest CI half-width improvements (Table~\ref{tab:kink-aware}), attacking D1---but as an online signal it is counterproductive.

\item \emph{C3 is decisive on discrete-inner systems (TB6).} The inventory MDP (I/DP/s) has no IPA channel, so all gradient-based solvers default to FD; PRIME reaches $0.096$ (tied with KG), while PRIME $-$C3 collapses to $30.55\pm23.32$ ($318\times$). The $(s,S)$ policy creates flat regions with zero FD gradient, separated by discrete jumps at integer thresholds: a single chain initialised at $\theta_0=12$ stalls, while multi-start places a chain near $\theta\approx9.5$. C3 therefore provides the spatial diversification that is essential when the objective has flat-gradient regions---a structural property of both feedback-amplified (TB2) and discrete-inner (TB6) SOSO systems.
\end{enumerate}

Taken together, the ablation evidence from TB2 and TB6 shows that C3 provides spatial diversification that neither C1's clean gradient nor C2's scale-invariant stepping can substitute for. Multi-start is essential when the SOSO landscape contains suboptimal attractors, which includes most agent-based models (ABMs) with feedback coupling and all systems with discrete inners (MILP, DP, MDP).

\begin{figure}[H]
\FIGURE{\includegraphics[width=0.82\textwidth]{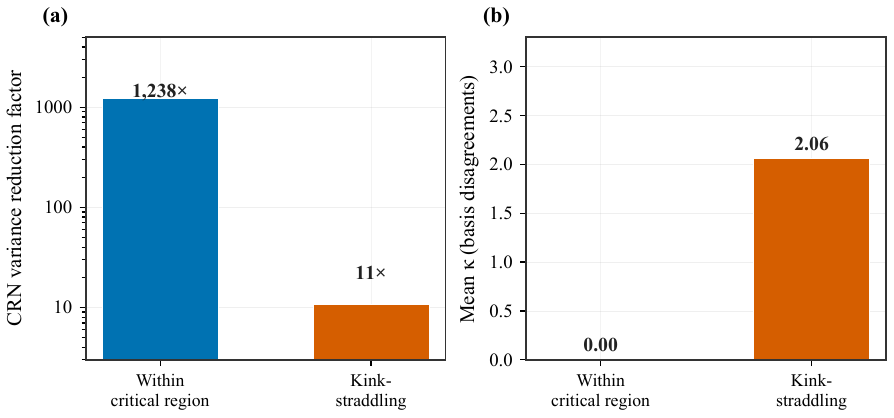}}
{(Result 4, TB1) CRN variance reduction and the basis-disagreement diagnostic $\kappa$ in the minimal market. Panel~(a): paired-difference variance reduction achieved by common random numbers---$1{,}238\times$ within a critical region ($\kappa{=}0$, $\theta_a{=}0.5$, $\theta_b{=}0.45$) versus $11\times$ across a basis-change kink ($\kappa{\approx}2$, $\theta_a{=}0.5$, $\theta_b{=}1.2$), where the basis change desynchronises the two paired simulation paths. Panel~(b): the corresponding mean $\kappa$ count (basis disagreements per paired replication) detected online from the two forward solves; the jump in $\kappa$ at the kink exactly tracks the collapse in CRN reduction, confirming Theorem~\ref{thm:crn}'s prediction that paired-difference variance decomposes into a within-region term ($\kappa{=}0$) and a $\kappa$-driven cross-kink term.}{5{,}000 paired replications per $(\theta_a,\theta_b)$ pair.}
\label{fig:crn}
\end{figure}

\begin{table}[H]
\caption{(Result 4, TB1) $\kappa$-aware vs.\ uniform replication allocation at fixed total budget. The $\kappa$-aware rule achieves slightly lower worst-case half-width by spending proportionally more replications on $\theta$ pairs where the basis-disagreement diagnostic $\kappa$ predicts inflated paired-difference variance. The modest absolute gain ($0.199$ vs.\ $0.203$) reflects that, on this testbed, kinks are sparse and most replicated pairs lie within a single critical region.}\label{tab:kink-aware}
\centering\footnotesize
\begin{tabular}{@{}lccc@{}}
\toprule
Allocation & Worst-case CI half-width & Mean CI half-width & Total reps \\
\midrule
$\kappa$-aware & $0.199$ & $0.118$ & $502$ \\
Uniform & $0.203$ & $0.118$ & $500$ \\
\bottomrule
\end{tabular}
\end{table}

\subsection{Result 5: Black-box gradient methods suffer exponential variance growth under feedback amplification, validating Theorem~\ref{thm:var}}\label{sec:res-r5}

Result~5 validates Theorem~\ref{thm:var}'s central prediction on the Beer Game (TB2, $g{=}2$, $m{=}4$): gradient variance grows as $\Theta(g^{2m})$ in the feedback depth for black-box estimators, while IPA bypasses this amplification by reading the basis sensitivity directly. Three sub-findings build from the failure mode of the strongest black-box baseline, through the mechanism, to the remedy.

\begin{enumerate}[leftmargin=*,itemsep=2pt,topsep=2pt]
\item \emph{Independent-FD-SA collapses on TB2.} The honest no-CRN black box achieves $314.6\pm116.5$ (Table~\ref{tab:so-comp})---two orders of magnitude worse than PRIME ($0.965$), with a standard deviation ($116.5$) larger than the gap of every other solver combined. On TB1 ($m{=}1$, no feedback amplification), Independent-FD-SA reaches $0.712$---worse but not catastrophically so. The contrast isolates feedback depth as the driver.

\item \emph{Mechanism: geometric amplification through differencing.} The state-sensitivity recursion $S_{t+1}=A_tS_t+b_t$ amplifies geometrically when $\|A_t\|\ge g>1$ (Theorem~\ref{thm:var}). Independent FD differences $[L(\theta{+}h)-L(\theta{-}h)]/(2h)$; each $L$ value sums amplified sensitivities from all upstream echelons, and the $1/h^{2}$ factor inflates the result. CRN-FD-SA mitigates this partially ($18.65$, $16.9{\times}$ better) but remains $19.3{\times}$ worse than PRIME: CRN cancels common-mode noise but not the $\Theta(g^{2m})$ geometric amplification.

\item \emph{Surrogate-as-decision breaks the loop.} Replacing the reactive forecast with a fixed target dampens the per-stage gain $g$ (Theorem~\ref{thm:eps}). As the surrogate reactivity dial $\alpha$ decreases from $1$ (fully reactive) to $0.2$, the per-replication IPA standard deviation drops from $15.2$ to $0.29$---a $51.7\times$ reduction (Table~\ref{tab:beergame})---and reaches the deterministic limit at $\alpha{=}0$. The error budget governs when this is safe: $\varepsilon{=}0$ exactly when the state carries no decision-relevant information (i.i.d.\ demand), while seasonal demand incurs $\varepsilon{=}19.5$ at the myopic horizon but collapses to $0$ once the weekly structure is visible ($H{\ge}2$, TB5). TB5det (the deterministic baseline driven by the historical SupplyGraph trajectory, \S\ref{sec:expt-design}) confirms this contrast: the myopic-horizon error $\varepsilon{=}20$ under seasonal demand drops to $\varepsilon{\approx}0$ when the same demand is shuffled to i.i.d., exactly as Theorem~\ref{thm:eps} predicts.
\end{enumerate}

The three sub-findings jointly confirm Theorem~\ref{thm:var} and its practical consequence: on any SOSO system with feedback gain exceeding unity, black-box gradient methods suffer variance that grows geometrically in the feedback depth. The planner must either switch to IPA (which reads the gradient from the basis) or deploy surrogate-as-decision (which dampens the gain). PRIME provides both options and the error budget to choose between them. Taken together, the surrogate-reactivity evidence (Table~\ref{tab:beergame}) and the black-box collapse on TB2 (Table~\ref{tab:so-comp}) establish that feedback amplification is not merely a theoretical concern but the dominant practical obstacle for gradient-based SO on coupled SOSO systems, and that the error-budgeted surrogate of Theorem~\ref{thm:eps} is the mechanism that restores tractability.

\begin{table}[H]
\caption{(Result 5, TB2) Surrogate reactivity $\alpha$ vs.\ IPA gradient standard deviation in the Beer Game (four-echelon, $\tau{=}2$). As $\alpha$ decreases from 1 (fully reactive) to 0.2, the per-replication IPA std drops from $15.2$ to $0.29$, a $51.7\times$ reduction, confirming Theorem~\ref{thm:var}'s prediction that the feedback gain $g$ propagates geometrically to gradient variance $\Theta(g^{2m})$.}\label{tab:beergame}
\centering\footnotesize
\begin{tabular}{@{}lrrrr@{}}
\toprule
$\alpha$ & IPA std dev & IPA mean & Cost mean & Std reduction vs.\ $\alpha{=}1$ \\
\midrule
1.0 (fully reactive) & $15.20$ & $16.96$ & $1030.0$ & $1\times$ \\
0.8 & $9.61$ & $13.38$ & $923.4$ & $1.6\times$ \\
0.6 & $5.35$ & $8.50$ & $785.5$ & $2.8\times$ \\
0.4 & $2.76$ & $3.14$ & $610.0$ & $5.5\times$ \\
0.2 & $0.29$ & $0.05$ & $407.3$ & $51.7\times$ \\
0.1 & ${\approx}0$ & ${\approx}0$ & $334.8$ & --- \\
0 (fully fixed) & ${\approx}0$ & ${\approx}0$ & $303.3$ & --- \\
\bottomrule
\end{tabular}
\end{table}

\section{Discussion}\label{sec:disc}

\subsection{Scope and positioning}\label{sec:scope}

Table~\ref{tab:positioning} positions the framework relative to adjacent paradigms. The defining property is that the inner layer is a \emph{non-equilibrium, path-dependent simulation} whose agents solve \emph{structured} optimizations---neither a closed-form equilibrium, nor an exact-inner multilevel program, nor an unstructured black box, nor a single optimization layer differentiated end-to-end.

\begin{table}[H]
\caption{Positioning relative to adjacent paradigms.}\label{tab:positioning}
\small
\centering
\begin{tabular}{@{}lll@{}}
\toprule
Paradigm & Inner layer & Distinction \\
\midrule
MPEC / equilibrium & closed-form equilibrium & non-equilibrium dynamics \\
Multilevel stochastic & exact inner, det.\ outer & approximate inner, sim.\ outer \\
MARL & inner is learning & inner is optimization \\
Black-box SO & simulation opaque & exploit LP basis/dual geometry \\
Metamodel SO & surrogate $J(\theta)$ & surrogate decision $a^{*}(s)$ \\
Differentiable opt.\ & one optimization layer & multi-period, multi-agent \\
\bottomrule
\end{tabular}
\end{table}

Beyond the three algorithmic components of PRIME (C1--C3, Table~\ref{tab:framework}), the paper provides two structural results that extend the framework's reach beyond what PRIME alone can deliver. Theorem~\ref{thm:crn} decomposes paired-difference variance under common random numbers into a within-region term ($\kappa{=}0$) and a basis-disagreement term driven by the free diagnostic $\kappa$ (count of agent-period cells where the two CRN-paired forward solves disagree on the active basis). This decomposition is not a component of PRIME---PRIME's chains run on independent RNG substreams, without CRN coupling---but it is directly useful to a practitioner who wants to deploy CRN pairing in a ranking-and-selection or confidence-interval context: it explains when CRN helps (inside a critical region, $\kappa{=}0$) and when it degrades (across kinks, $\kappa>0$), and it drives a $\kappa$-aware replication allocation (Table~\ref{tab:kink-aware}) that trims the worst-case CI half-width at fixed budget. Theorem~\ref{thm:var} and Theorem~\ref{thm:eps} together characterize when black-box gradient methods fail under feedback amplification ($\Theta(g^{2m})$ variance growth) and when surrogate-as-decision dampens the gain safely (the error budget $\varepsilon$). These are design-time tools: a practitioner uses them to decide, before running PRIME, whether to deploy the full PRIME solver (IPA + adaptive step + multi-start), whether to dampen the feedback gain via surrogate decisions, or whether to fall back to finite differences with the error budget.

\subsection{Practical guidance}\label{sec:ipa-vs-fd}

Theorems~\ref{thm:ipa} and \ref{thm:var} together draw a practical boundary, and the experimental evidence of Section~\ref{sec:expts} establishes the following decision rules for IPA-based simulation optimization of SOSO systems.
\begin{enumerate}
\item \emph{C1 (IPA) and C2 (adaptive step) are required}---they provide the core gradient engine and scale-invariant stepping. IPA is preferable to finite differences when the inner is an LP with a non-degenerate basis and the feedback gain $g$ is bounded (e.g., single-agent or weakly coupled systems); IPA's ${\sim}2000{\times}$ gradient-variance advantage (Table~\ref{tab:e1}) is decisive on feedback-coupled systems where FD's $1/h^{2}$ noise amplification interacts destructively with the geometric amplification of Theorem~\ref{thm:var}.
\item \emph{C3 (multi-start ensemble) is strongly recommended} whenever the objective landscape may contain suboptimal attractors---which includes most ABMs with feedback coupling. Theorem~\ref{thm:kink} guarantees kinks are measure-zero in policy space (so fixed start points almost surely avoid them), and Theorem~\ref{thm:crn}'s $\kappa$ diagnostic predicts low basis-disagreement variance near the optimum, so at least one equispaced chain lands in a well-behaved region; together they provide spatial diversification that avoids regions where gradient signal vanishes or CRN reduction degrades. The C1+C2+$\kappa$ baseline (Table~\ref{tab:so-comp}) confirms that $\kappa$-based step damping is \emph{not} a substitute: multi-start spatial diversification is the robust mechanism.
\item \emph{Surrogate-as-decision is the remedy when feedback gain exceeds unity.} When the feedback chain is deep ($m\ge3$) and the gain is large ($g\ge2$), even IPA's structural gradient suffers heavy-tailed variance (Limitation~2, \S\ref{sec:limitations}). The treatment follows Theorems~\ref{thm:var}--\ref{thm:eps}: dampen the feedback gain $g$ by reducing surrogate reactivity (Table~\ref{tab:beergame}), then bound the decision error $\varepsilon$ and verify it lies inside the planner's indifference zone. This is a framework-level lever, not a PRIME component (Table~\ref{tab:framework} note).
\item \emph{PRIME $-$C1 (FD+C2+C3) is a viable fallback} on large-scale industrial LPs where shadow prices exhibit high replication-to-replication variability (e.g., TB4's JD.com chain, where the FD ablation modestly outperforms full PRIME; Result~2, \S\ref{sec:res-r2}). On feedback-amplified ABMs, IPA's variance advantage is decisive and the fallback is not recommended.
\end{enumerate}

The cross-over point, characterized by Theorem~\ref{thm:var}'s $\Theta(g^{2m})$ growth, is empirically at $m\approx2$--$3$ for $g\approx2$. The practical implication of Result~\ref{sec:res-r5} is that on any SOSO system with feedback coupling where the feedback gain exceeds unity, black-box gradient methods (FD, SPSA) will suffer variance that grows geometrically in the feedback depth, and the planner must either switch to IPA (which reads the gradient from the basis, bypassing the amplification) or deploy surrogate-as-decision (which dampens the gain itself). The PRIME framework provides both options and the error budget to choose between them.

\subsection{Limitations}\label{sec:limitations}

We discuss each limitation candidly, including what it means for the applicability of our framework.

\begin{enumerate}
\item \emph{Transversality (A2) can fail on degenerate parameters.} If the LP has multiple optimal bases at some state (e.g., a constraint is redundant), the basis changes discontinuously and the IPA sensitivity may be undefined. In practice, degeneracy occurs when capacity constraints are exactly saturated or when the demand caps are exactly equal. The SOSO framework assumes (A2) that this is a measure-zero event. When it is not, a lexicographic RHS perturbation restores transversality at the cost of an $O(\varepsilon_{\mathrm{lex}})$ perturbation that vanishes in the limit. We have not verified (A2) for all practical models; the user must check it per application. MILP inners, due to the presence of integer variables, are particularly prone to degeneracy.

\item \emph{IPA has heavy tails in feedback-amplifying systems.} At full reactivity ($\alpha{=}1$, Beer Game), the per-replication IPA gradient at depth $m{=}4$ exhibits a heavy right tail: a Hill-type estimator applied to the replication-level gradient samples (the same samples underlying Table~\ref{tab:beergame}) gives a tail index in the range $\sim2$--$4$, so the variance is finite but dominated by rare runaway sensitivity paths where the state sensitivity $S_t$ grows abnormally. This means the \emph{sample mean} of IPA estimates can converge slowly (the central limit theorem holds but the effective sample size is reduced). The error-budgeted surrogate of Theorem~\ref{thm:eps} dampens the feedback gain and lightens the tails (Table~\ref{tab:beergame}), but we have not characterized the tail behavior theoretically; the index $\sim2$--$4$ is an empirical observation from the TB2 replication data, not a proven bound, and its precision is limited by the number of replications. If the tail index were below 2 (infinite variance), IPA might not converge in distribution.

\item \emph{The $\kappa$-jump bound (Theorem~\ref{thm:crn}) is not tight.} Our inequality $\mathrm{Var}(D)\le 2\bigl(\mathrm{Var}(D)|_{\kappa=0}+\bar\Delta^2\mathbb{E}[\kappa^2]\bigr)$ bounds the \emph{worst-case} per-cell jump by $\bar\Delta$ (the universal factor $2$ absorbs the smooth/kink cross-covariance). In practice the jump magnitudes vary (some basis changes cause larger action discontinuities than others), and a tighter bound would use the actual jump magnitudes rather than the worst-case $\bar\Delta$. The practical impact is that the $\kappa$ count overestimates the variance inflation in the presence of small jumps, but this is conservative, not optimistic.

\item \emph{No exact IPA for MILP or discrete MDP inners.} When the inner problem has integer variables, the active set is ill-defined and the basis is not unique. When the inner is a discrete MDP (TB6), the policy is a step function. In both cases, the IPA gradient channel (C1) is unavailable; the planner falls back to finite differences (PRIME $-$C1, Table~\ref{tab:framework}) and uses the error budget of Theorem~\ref{thm:eps} to bound the outer-optimum shift introduced by any surrogate inner policy. The continuous-DP case (TB3) shows that smoothness of the inner policy is the key, not the inner type; future work on smooth relaxations of discrete policies could bridge this gap.

\item \emph{Scope of the real-data experiment (TB5).} TB5 drives the rolling-horizon production LP with the aggregate daily demand of the 41 SupplyGraph SKUs \citep{wasi2024supplygraph}, rather than optimizing the full multi-echelon distribution network. The theorems are agnostic to the data source, so this scoping choice is illustrative of real demand structure rather than a full-scale deployment study.
\end{enumerate}

\section{Conclusion}\label{sec:conc}

We have shown that Simulation-Optimization of Systems of Optimizers---agent-based simulations in which every agent solves a structured optimization at every epoch---need not be treated as slow black-box simulations. Their inner optimization geometry, specifically the LP basis and dual variables, propagates upward through the dynamics and yields three capabilities: a cheap, exact, single-replication IPA gradient of the outer objective (Theorems~\ref{thm:ipa}--\ref{thm:kink}), a CRN variance-reduction decomposition governed by basis-disagreement count $\kappa$ (Theorem~\ref{thm:crn}), and a unified error budget for surrogate decisions that governs how much computation can be offloaded from exact inner solves (Theorems~\ref{thm:var}--\ref{thm:eps}).

On a real-world supply-chain dataset with 1{,}000 SKUs across six distribution centers (1 RDC + 5 FDCs), the IPA gradient---read from the LP's shadow price of transfer capacity---traces the marginal-value curve for capacity and locates the saturation point at which additional capacity no longer justifies its cost, while common random numbers achieve variance reductions exceeding $10{,}000\times$ over independent streams. Assembled into the PRIME solver, the framework achieves the best or tied-best optimality gap on all six testbeds in a head-to-head comparison against classic SO baselines (finite-difference SA, SPSA, stochastic Kriging, Knowledge Gradient), with near-zero oscillation convergence (Table~\ref{tab:stability}, Result~1)---a direct consequence of IPA's ${\sim}2000\times$ per-replication gradient-variance advantage at $\theta{=}0.5$ in the minimal market (Table~\ref{tab:e1}, Result~2)---while black-box gradient methods suffer the exponential variance explosion of Theorem~\ref{thm:var} on feedback-amplifying systems (Result~5). Ablation experiments confirm that each component of PRIME is essential: the adaptive step (C2, Result~3) is the most critical gate on feedback-coupled systems through which the IPA gradient (C1, Result~2) passes, and the multi-start ensemble (C3, Result~4) provides the spatial diversification needed to escape suboptimal attractors in both feedback-coupled and discrete-inner systems. These results indicate that exploiting embedded-optimization structure is a practically useful direction for simulation optimization, with relevance for supply-chain management, electricity markets, and other domains whose agents solve LP or smooth-DP inners of the kind studied here (with MILP as an open direction, \S\ref{sec:limitations}).

Several directions merit further investigation. First, \emph{adaptive surrogate placement}: learn which agents or periods to surrogate from online diagnostics (basis stability, condition number, $\kappa$), turning surrogate-as-decision into a budget-allocation procedure with a formal indifference-zone guarantee. Second, \emph{MILP and coupled-DP inner problems}: the four taxonomy cells currently marked as future work in Figure~\ref{fig:soso-taxonomy}---I/MILP/d, I/MILP/s, C/MILP/s, and C/DP/s---are open. For MILP inners, the error-budget framework extends through LP-relaxation surrogates (Remark~\ref{rmk:delta}), but exact IPA requires a sensitivity theory for integer active sets; the C/MILP/s and C/DP/s cells are particularly challenging because coupled feedback through integer or discrete decisions amplifies discretely rather than geometrically. Third, \emph{tail characterization for feedback-amplified IPA}: the empirical tail-index estimate $\sim2$--$4$ (Limitation~2, \S\ref{sec:limitations}) suggests that feedback depth governs the heavy-tailedness of the IPA gradient; a theoretical characterization would sharpen the practical boundary between IPA and surrogate-as-decision. Fourth, \emph{integration with modern SO solvers}: plugging the IPA gradient into stochastic approximation, Bayesian optimization, or R\&S procedures, with convergence-rate comparisons against black-box baselines.

\section{Code and Data Disclosure}\label{sec:codedata}

The Python implementation of all algorithms studied in this paper---the proposed IPA/CRN/surrogate framework and every baseline (finite differences with and without CRN, full-exact inner solves, LP-horizon surrogates, truncated value iteration, LP-relaxation surrogates)---together with the instance generators, the benchmark data sets, and a README giving full reproduction instructions for every table and figure, will be deposited in a public repository upon publication. To preserve double-anonymous review, the repository URL and authorship metadata are withheld from this version; the same materials are available to the editor and reviewers on request during review. The synthetic instances are deterministic functions of a random seed and can be regenerated exactly; the real-world SupplyGraph and JD GOC data sets are publicly available from their respective sources. No data or code exemption is requested.

\ACKNOWLEDGMENT{This work was supported in part by the National Natural Science Foundation of China (grant number 72271227), the Youth Innovation Promotion Association CAS (grant number 110800EAG2), the Xiaomi Young Talents Program, and the MOE Social Science Laboratory of Digital Economic Forecasts and Policy Simulation at UCAS.}

\newpage
\begin{center}
{\Large\bfseries E-Companion}
\end{center}
\medskip

\setcounter{section}{0}
\renewcommand{\thesection}{EC.\arabic{section}}
\renewcommand{\thesubsection}{EC.\arabic{section}.\arabic{subsection}}
\setcounter{theorem}{0}
\renewcommand{\thetheorem}{EC.\arabic{theorem}}
\setcounter{remark}{0}
\renewcommand{\theremark}{EC.\arabic{remark}}
\setcounter{table}{0}
\renewcommand{\thetable}{EC.\arabic{table}}
\providecommand{\theHsection}{}
\providecommand{\theHsubsection}{}
\providecommand{\theHtheorem}{}
\providecommand{\theHremark}{}
\providecommand{\theHtable}{}
\renewcommand{\theHsection}{EC.\arabic{section}}
\renewcommand{\theHsubsection}{EC.\arabic{section}.\arabic{subsection}}
\renewcommand{\theHtheorem}{EC.\arabic{theorem}}
\renewcommand{\theHremark}{EC.\arabic{remark}}
\renewcommand{\theHtable}{EC.\arabic{table}}

This e-companion provides the complete proofs of the five theorems and one remark stated in the main paper (Sections~EC.1--EC.2), together with a glossary of abbreviations (Section~EC.3). Each theorem is restated here for self-containment, then proved. Assumptions (A1)--(A5) and the notation $\Theta^{\circ}$, $\kappa$, $\bar\Delta$, $\widehat{J}'_{\mathrm{IPA}}$, etc.\ are as defined in the main paper.

\section{Proofs of the main results}\label{ec:proofs}

\begin{theorem}[IPA unbiasedness]\label{ecthm:ipa}
Under (A1)--(A5), for every $\theta\in\Theta^{\circ}$, $\mathbb{E}_\xi[\widehat{J}'_{\mathrm{IPA}}]=\nabla J(\theta)$.
\end{theorem}

\begin{proof}{Proof.}
We verify the three conditions for exchanging the derivative and the expectation: (i) pathwise differentiability, (ii) a dominating function, (iii) integrability.

\emph{Step 1: Critical-region structure.} By parametric-LP theory \citep{gal1995,bazaraa2010}, the $(s,\theta)$-space is partitioned into finitely many polyhedral critical regions $\{P_k\}_{k=1}^K$. On each region $P_k$, the optimal basis $B_k$ (the set of $n$ active constraint rows) is constant, and the primal solution is the affine map $a^{*}(s,\theta)=B_k^{-1}h_k(s,\theta)$, where $h_k$ collects the RHS of the active rows. The active-row dual is $\lambda^{*}_k=c_k^{\top}B_k^{-1}$. Since $G$ is constant (A1) and $h$ is affine (A1), $a^{*}$ is affine in $(s,\theta)$ on $\operatorname{int}(P_k)$.

\emph{Step 2: Pathwise differentiability.} Define $E_\theta$ as the event that the orbit $\{(s_{i,t}(\theta,\xi),\theta)\}_{i,t}$ stays in region interiors for all $(i,t)$. By (A3), $\mathbb{P}(E_\theta)=1$. On $E_\theta$, each $a^{*}_{i,t}$ is locally affine in $(s,\theta)$, so the state recursion $s_{i,t+1}=\phi_i(s_{i,t},a^{*}_{i,t},\xi_t,\theta)$ is a $C^1$ composition of $C^1$ maps (by A4), and the reward $\pi_i$ is $C^1$ (A4). Therefore $L(\cdot,\xi)$ is $C^1$ in a neighborhood of $\theta$, with pathwise derivative exactly $\widehat{J}'_{\mathrm{IPA}}(\theta,\xi)$ obtained by the chain rule applied to the forward-sensitivity recursions---the standard IPA construction of \citet{glasserman1991}.

\emph{Step 3: A pathwise dominating function.} Expanding the chain rule on the interior of a critical region, the action and state sensitivities satisfy $\partial_\theta a^{*}_{i,t}=B_k^{-1}(\partial_\theta h_k+\partial_s h_k\,S_{i,t})$ and $S_{i,t+1}=A_{i,t}S_{i,t}+b_{i,t}$ with $A_{i,t}=\partial_s\phi_i+\partial_a\phi_i\,B_k^{-1}\partial_s h_k$ and $b_{i,t}=\partial_a\phi_i\,B_k^{-1}\partial_\theta h_k+\partial_\theta\phi_i$. By (A4) the derivatives of $c,\phi_i,\pi_i$ are bounded, by (A1) those of $h$ are constant, and the finitely many bases give $B_{\max}:=\max_k\|B_k^{-1}\|<\infty$; hence $\|A_{i,t}\|\le g_0$ and $\|b_{i,t}\|\le c_0$ for $\xi$- and state-independent constants $g_0,c_0$. Unrolling over the finite horizon from $S_{i,0}=0$, $\|S_{i,t}\|\le c_0\sum_{\tau=0}^{t-1}g_0^{t-1-\tau}\le C_T$ for a finite constant $C_T$ (depending on $T$ but not on $\xi$ or $s$). Substituting into $\partial_\theta L=\sum_{i,t}[\partial_a\pi_i\,\partial_\theta a^{*}_{i,t} +\partial_s\pi_i\,S_{i,t}+\partial_\theta\pi_i]$ yields the pathwise bound
\begin{equation}\bigl|\partial L/\partial\theta(\theta,\xi)\bigr|\le
G:=C'\quad\text{a.s.}\end{equation}
for a finite constant $C'$ (depending on $T,N,B_{\max}$ and the derivative bounds). Crucially the bound involves the \emph{sensitivity} $S_{i,t}$, which is uniformly bounded for finite $T$, rather than the state norm itself.

\emph{Step 4: Integrability of $J$ and of the bound.} Since $G=C'$ is a finite constant, $\mathbb{E}_\xi G<\infty$ trivially. Assumption (A5) is used separately to ensure the outer objective is finite: because $a^{*}=B_k^{-1}h_k(s,\theta)$ grows at most linearly in $\|s\|$ and $\pi_i$ is $C^1$ (A4), $|L(\theta,\xi)|\le C''\bigl(1+\sum_{i,t}\|s_{i,t}\|\bigr)$, so (A5) (via Cauchy--Schwarz and finiteness of $N\!\cdot\!T$) gives $\mathbb{E}_\xi|L|<\infty$ and hence $J(\theta)=\mathbb{E}_\xi L$ is well defined.

\emph{Step 5: Exchange.} Set $\widehat{J}'_{\mathrm{IPA}}=0$ on $\Omega\setminus E_\theta$ (a null set), so $\partial L/\partial\theta=\widehat{J}'_{\mathrm{IPA}}$ a.s. By the mean value theorem the difference quotient satisfies $\bigl|[L(\theta{+}h,\xi)-L(\theta,\xi)]/h\bigr|\le\sup_{\theta'\in[\theta,\theta+h]} |\partial L/\partial\theta(\theta',\xi)|\le G$ uniformly in small $h$, and this bound is $\xi$-integrable. The dominated convergence theorem then exchanges the derivative and the expectation:
$$\nabla J(\theta)=\nabla_\theta\mathbb{E}_\xi[L(\theta,\xi)]
=\mathbb{E}_\xi\bigl[\partial L/\partial\theta\bigr]
=\mathbb{E}_\xi[\widehat{J}'_{\mathrm{IPA}}].$$
This completes the proof.\Halmos
\end{proof}

\begin{theorem}[Critical events are measure-zero]\label{ecthm:kink}
Under (A2)--(A4), $\Theta\setminus\Theta^{\circ}$ has Lebesgue measure zero.
\end{theorem}

\begin{proof}{Proof.}
A basis change at $(i,t)$ occurs iff the orbit crosses a critical-region boundary $\partial P_k$, i.e., iff some tightness residual $r_{i,t,j}(\theta,\xi):=$ (slack of constraint $j$ at the current basis) vanishes. There are finitely many such residuals across all $(i,t,j)$; stack them into $f_\xi:\Theta\to\mathbb{R}^M$. Each $r_{i,t,j}$ is $C^1$ in $\theta$ (by A4 and the parametric-LP structure), so $f_\xi$ is $C^1$.

By (A2), $0$ is a regular value of $f_\xi$ for $\mathbb{P}$-a.e.\ $\xi$, meaning $D_\theta f_\xi$ is surjective at every preimage point. By the regular value theorem, $f_\xi^{-1}(\{0\})$ is a smooth submanifold of dimension $d-M$ (when nonempty). In our setting $d\le M$, so $d-M\le 0$: the preimage is either empty (when $d<M$) or a zero-dimensional submanifold, i.e.\ a discrete set of points (when $d=M$); in either case it has $d$-dimensional Lebesgue measure zero. In the special case $d=1$ (a one-dimensional policy parameter, as used in all numerical experiments of this paper), each zero of the residual map is isolated (the linearisation dominates the little-o remainder), giving a closed discrete---hence countable and measure-zero---zero set. The general $d\le M$ case follows from the regular value theorem as stated above. (Sard's theorem is not needed here: it controls the set of critical \emph{values} in the codomain, whereas (A2) already makes $0$ a regular value, so the regular value theorem applies directly; only $C^1$ regularity, guaranteed by (A4), is required. Finally, the genericity of the regular-value condition (A2) itself follows from Thom transversality; the exceptional non-transverse parameters are handled by a lexicographic RHS perturbation.\Halmos
\end{proof}

\noindent\textit{Scope remark.} The statement above holds for a general low-dimensional policy $\theta\in\mathbb{R}^d$ with $d\le M$; its proof rests on the regular value theorem, a standard result of differential topology. All numerical experiments in this paper use the scalar case $d=1$, for which the zeros of the $C^1$ residual map are isolated and the measure-zero conclusion follows from elementary real analysis (closed, discrete, hence countable, hence null). The $d>1$ generality is therefore a theoretical completeness statement; it is not exercised by any testbed reported here.

\begin{theorem}[CRN covariance decomposition]\label{ecthm:crn}
$\mathrm{Var}(D)\le 2\bigl(\mathrm{Var}(D)|_{\kappa=0}+\bar{\Delta}^2\,\mathbb{E}[\kappa^2]\bigr)$, where $D=L(\theta_a,\xi)-L(\theta_b,\xi)$, $\kappa$ counts basis-disagreement cells and $\bar{\Delta}$ bounds reward jumps. The universal factor $2$ absorbs the smooth/kink cross-covariance and does not affect the qualitative conclusion: CRN helps within a region ($\kappa=0$) and degrades with $\kappa$.
\end{theorem}

\begin{proof}{Proof.}
The variance identity $\mathrm{Var}(D)=\mathrm{Var}\,L(\theta_a)+\mathrm{Var}\,L(\theta_b) -2\,\mathrm{Cov}(L(\theta_a),L(\theta_b))$ follows from expanding $\mathrm{Var}(X-Y)$. Under CRN (shared $\xi$), the covariance $\mathrm{Cov}(L(\theta_a),L(\theta_b))$ is positive because the shared shocks drive both members through the same smoothed dynamics.

For the decomposition, write $L(\theta,\xi)=L^{\mathrm{sm}}(\theta,\xi) +L^{\mathrm{kk}}(\theta,\xi)$, where $L^{\mathrm{sm}}$ is the smooth part (within-region contribution) and $L^{\mathrm{kk}}$ collects the jumps at basis changes along the orbit. A jump at cell $(i,t)$ has magnitude bounded by $\bar{\Delta}=L_\pi\cdot\operatorname{diam}(\mathcal{A})$, where $L_\pi$ is the reward's Lipschitz constant (A4) and $\operatorname{diam}(\mathcal{A})$ is the bounded primal diameter (from the finitely many $\|B_k^{-1}\|$ and bounded $\bar h$).

On a cell where the two $\theta$'s share the same basis ($\kappa=0$ there), the jump is common-mode: it occurs at the same $(i,t)$ with the same magnitude, so it cancels in $D=L(\theta_a)-L(\theta_b)$. On a basis-disagreement cell ($\kappa>0$), one member's action jumps where the other's does not, so the jump survives in $D$. Let $D^{\mathrm{kk}}(\xi)$ denote the sum of the surviving jumps and $D^{\mathrm{sm}}:=D-D^{\mathrm{kk}}=D|_{\kappa=0}$ the smooth remainder, so $D=D^{\mathrm{sm}}+D^{\mathrm{kk}}$. Each surviving jump has magnitude at most $\bar{\Delta}$ and there are $\kappa(\xi)$ of them, so $|D^{\mathrm{kk}}(\xi)|\le\bar{\Delta}\,\kappa(\xi)$, whence $\mathrm{Var}_\xi(D^{\mathrm{kk}})\le\mathbb{E}_\xi[(D^{\mathrm{kk}})^2]\le \bar{\Delta}^{2}\,\mathbb{E}_\xi[\kappa(\xi)^2]$, with no assumption that the per-cell jumps be uncorrelated. To combine without controlling the sign of the smooth/kink cross-covariance, apply $(a+b)^2\le 2a^2+2b^2$ to the centered variables $D-\mathbb{E}D=(D^{\mathrm{sm}}-\mathbb{E}D^{\mathrm{sm}}) +(D^{\mathrm{kk}}-\mathbb{E}D^{\mathrm{kk}})$:
\begin{align*}
\mathrm{Var}_\xi(D)
&\le 2\,\mathrm{Var}_\xi(D^{\mathrm{sm}})+2\,\mathrm{Var}_\xi(D^{\mathrm{kk}})\\
&\le 2\,\mathrm{Var}_\xi(D)\big|_{\kappa=0}
+2\,\bar{\Delta}^{2}\,\mathbb{E}_\xi[\kappa(\xi)^2],
\end{align*}
which is the claimed bound; the cross-term $2\,\mathrm{Cov}(D^{\mathrm{sm}},D^{\mathrm{kk}})$ is absorbed by the prefactor $2$, so no assumption on its sign is required.\Halmos
\end{proof}

\begin{theorem}[Variance amplification]\label{ecthm:var}
If the state-sensitivity recursion $S_{t+1}=A_tS_t+b_t$ has per-stage gain $\|A_t\|\le g$ with $g>1$ over $m$ feedback stages, with non-cancelling coherent gains (the scalar case $A_t\ge g>0$ suffices) and integrable shocks, then the IPA-gradient second moment satisfies $\mathbb{E}_\xi|\widehat{J}'_{\mathrm{IPA}}|^2=\Theta(g^{2m})$ in the feedback depth $m$; consequently $\mathrm{Var}[\widehat{J}'_{\mathrm{IPA}}]=\Theta(g^{2m})$ whenever the mean $\nabla J(\theta)$ does not cancel it. The CRN-FD variance grows only polynomially in $m$.
\end{theorem}

\begin{proof}{Proof.}
Unrolling the state-sensitivity recursion $S_{t+1}=A_tS_t+b_t$ with $S_0=0$:
$$S_m=\sum_{j=0}^{m-1}\Bigl(\prod_{\ell=j+1}^{m-1}A_\ell\Bigr)b_j.$$

\emph{Upper bound.} With $\|A_\ell\|\le g$ and $\|b_j\|\le c\|\xi_j\|$:
$$\|S_m\|\le\sum_{j=0}^{m-1}g^{m-1-j}\,c\|\xi_j\|
\le c\,g^{m-1}\sum_{j=0}^{m-1}g^{-j}\|\xi_j\|
\le c\,g^{m-1}\cdot\frac{g}{g-1}\cdot\max_j\|\xi_j\|
=O\bigl(g^{m}\max_j\|\xi_j\|\bigr).$$

\emph{Lower bound.} Under non-cancelling positive gains (scalar $A_\ell\ge g>0$), the $j{=}0$ term is $\bigl(\prod_{\ell=1}^{m-1}A_\ell\bigr)b_0\ge g^{m-1}\|b_0\|$. The remaining terms are $O(g^{m-2})$ and thus dominated for $g>1$. Hence $\|S_m\|\ge c'\,g^{m}\|\xi_0\|$ for some $c'>0$, whenever $b_0\ne0$ (generic).

\emph{Combining.} The IPA estimator is $\widehat{J}'_{\mathrm{IPA}}=\sum_t\langle w_t,S_t\rangle$ with weights $w_t=\partial_a\pi_{i(t)}$ (the reward gradients, bounded by A4). The upper and lower bounds on $\|S_t\|$ give $|\widehat{J}'_{\mathrm{IPA}}|^{2}=\Theta(g^{2m}\max_j\|\xi_j\|^{2})$, hence the second moment $\mathbb{E}|\widehat{J}'_{\mathrm{IPA}}|^{2} =\Theta\bigl(g^{2m}\,\mathbb{E}\max_j\|\xi_j\|^{2}\bigr)$. For sub-Gaussian shocks $\mathbb{E}\max_j\|\xi_j\|^{2}=O(\log m)$, polynomially bounded in $m$ and thus negligible next to the exponential $g^{2m}$ factor, so $\mathbb{E}|\widehat{J}'_{\mathrm{IPA}}|^{2}=\Theta(g^{2m})$ in $m$. Since IPA is unbiased, $\mathrm{Var}[\widehat{J}'_{\mathrm{IPA}}] =\mathbb{E}|\widehat{J}'_{\mathrm{IPA}}|^{2}-|\nabla J(\theta)|^{2}$, which is $\Theta(g^{2m})$ whenever $|\nabla J(\theta)|$ does not itself grow at rate $g^{m}$---in particular at any $\theta$ with bounded $\nabla J(\theta)$, e.g.\ away from the regime where the mean sensitivity amplifies.

\emph{FD growth.} CRN-FD differences the realized cost $L$, which depends on the trajectory only through the realized states $s_t$ and never propagates the sensitivity $S_t=\partial s_t/\partial\theta$; its variance is therefore governed by the state dynamics $\phi$ rather than by the amplifying sensitivity recursion. When the cost channel does not itself amplify with feedback depth---as in the Beer Game, where the bullwhip lives in the reactive order forecast while the holding/backlog cost stays bounded-variance---$\mathrm{Var}(L)$, and hence the CRN-FD variance, grows only polynomially in $m$.\Halmos
\end{proof}

\begin{theorem}[Error-budget characterization]\label{ecthm:eps}
$\varepsilon\ge0$, and for a static surrogate $\varepsilon=0$ iff the optimal action is a.s.\ state-independent. Agnostic to inner type (LP, MILP, DP/RL).
\end{theorem}

\begin{proof}{Proof.}
\emph{Non-negativity.} Since $a^{*}_{\mathrm{ex}}(s)$ is the exact maximizer of $\pi(\cdot,s)$, we have $\pi(a^{*}_{\mathrm{ex}}(s),s)\ge\pi(a^{*}_{\mathrm{su}}(s),s)$ pointwise for every $s$. Taking expectation over $s$ gives $\varepsilon=\mathbb{E}_s[\pi(a^{*}_{\mathrm{ex}})-\pi(a^{*}_{\mathrm{su}})]\ge0$.

\emph{Characterization of $\varepsilon=0$.} For a static surrogate $a^{*}_{\mathrm{su}}(s)\equiv a_0$:
$$\varepsilon=\mathbb{E}_s\bigl[\max_a\pi(a,s)\bigr]-\mathbb{E}_s[\pi(a_0,s)]
=\mathbb{E}_s\bigl[\max_a\pi(a,s)-\pi(a_0,s)\bigr].$$
This equals $0$ iff $\max_a\pi(a,s)=\pi(a_0,s)$ for $\mathbb{P}_s$-a.e.\ $s$ (since the integrand is nonnegative). That is, $a_0$ is optimal for a.e.\ $s$, meaning the argmax $\arg\max_a\pi(a,s)$ is a.s.\ constant---the optimal decision does not depend on the realized state. This state-independence is the operative condition; it is implied by, but not equivalent to, the iid/martingale-difference condition $\mathcal{L}(s_{t+1}\mid s_t)=\mathcal{L}(s_{t+1})$, which is one common sufficient case (a state whose conditional law carries no information beyond the marginal).

\emph{Agnosticism to inner type.} The argument uses only $\pi$ and the state law $\mathbb{P}_s$, not the algorithm computing $a^{*}_{\mathrm{ex}}$. Hence it applies whether the exact decision comes from an LP, MILP, or DP/RL solver.\Halmos
\end{proof}

\begin{remark}[Outer-optimum shift is bounded by $\varepsilon$]\label{ecrmk:delta}
If $|J(\theta,a^{*}_{\mathrm{ex}})-J(\theta,a^{*}_{\mathrm{su}})|\le\varepsilon_{\max}$ for all $\theta$, then adopting the surrogate's optimum costs at most $\delta:=J(\theta^{*}_{\mathrm{su}})-J(\theta^{*}_{\mathrm{ex}})\le2\varepsilon_{\max}$ in true objective. When the surrogate is one-sided optimistic (as an LP relaxation of a minimization is), the bound tightens to $\delta\le\varepsilon_{\max}$.
\end{remark}

\begin{proof}{Proof.}
From $|f-g|\le\varepsilon$: $f\le g+\varepsilon$ and $g\le f+\varepsilon$. With $f=J(\cdot,a^{*}_{\mathrm{ex}})$, $g=J(\cdot,a^{*}_{\mathrm{su}})$, $\theta^{*}_{\mathrm{ex}}=\arg\max f$, $\theta^{*}_{\mathrm{su}}=\arg\max g$:
\begin{align*}
f(\theta^{*}_{\mathrm{ex}}) &\le g(\theta^{*}_{\mathrm{ex}})+\varepsilon
   \le g(\theta^{*}_{\mathrm{su}})+\varepsilon
   \le f(\theta^{*}_{\mathrm{su}})+2\varepsilon,
\end{align*}
so $f(\theta^{*}_{\mathrm{ex}})-f(\theta^{*}_{\mathrm{su}})\le 2\varepsilon_{\max}$. For a one-sided optimistic surrogate ($g\le f$ pointwise), the last inequality tightens to $g(\theta^{*}_{\mathrm{su}})\le f(\theta^{*}_{\mathrm{su}})$, giving $\delta\le\varepsilon_{\max}$.\Halmos
\end{proof}

\section{Supporting verification of assumption (A5)}\label{ec:a5}

Assumption (A5) ($\mathbb{E}\sum_{i,t}\|s_{i,t}\|^{2}<\infty$) is verified per testbed as follows. For the controlled testbeds TB1--TB2 the state spaces are bounded by construction (finite capacity caps, finite demand distribution with light tails), so the sum is trivially finite. For the real-data testbeds TB5 (SupplyGraph) and TB4 (JD.com), the inventory states are bounded by the finite planning horizon and the bounded empirical demand distributions; TB6 (inventory MDP) has a finite state space by construction; TB3 (continuous LQ) has Gaussian state of finite variance (the LQ regulator is stable). In every case $\mathbb{E}\sum_{i,t}\|s_{i,t}\|^{2}$ is finite for finite $T$, satisfying (A5).

\section{Glossary of abbreviations}\label{ec:glossary}
Table~\ref{ectab:glossary} lists every abbreviation used in the paper.

\begin{table}[H]
\caption{Abbreviations used in this paper.}\label{ectab:glossary}
\small
\centering
\begin{tabular}{@{}lll@{}}
\toprule
Abbreviation & Full name & First use \\
\midrule
ABM & Agent-based model & \S5.4 \\
CRN & Common random numbers & \S1.3 \\
DCT & Dominated convergence theorem & \S3.2 \\
DP & Dynamic program & \S1 \\
EWMA & Exponentially weighted moving average & \S4.2 \\
FD & Finite differences & \S1.4 \\
FDC & Fulfillment distribution center & \S1.4 \\
FD-SA & Finite-difference stochastic approximation & \S5.1 \\
GOC & (JD) Global Optimization Challenge & \S5.1 \\
IPA & Infinitesimal perturbation analysis & \S1.3 \\
KG & Knowledge Gradient & \S5.1 \\
KKT & Karush--Kuhn--Tucker & \S1.3 \\
LP & Linear program & \S1 \\
LQ & Linear-quadratic & \S2.2 \\
LQR & Linear-quadratic regulator & \S2.2 \\
MARL & Multi-agent reinforcement learning & \S1.3 \\
MDP & Markov decision process & \S2.2 \\
MILP & Mixed-integer linear program & \S1.1 \\
MPEC & Mathematical program with equilibrium constraints & \S1.3 \\
PRIME & Primal-IPA with Robust Multi-start \& Ensemble (the proposed solver) & \S4 \\
RDC & Regional distribution center & \S1.4 \\
RHS & Right-hand side & \S1.3 \\
RL & Reinforcement learning & \S1.1 \\
RNG & Random number generator & \S4.3 \\
SK & Stochastic Kriging & \S5.1 \\
SKU & Stock-keeping unit & \S1.4 \\
SO & Simulation optimization & \S1 \\
SOSO & Simulation-Optimization of Systems of Optimizers & \S1.1 \\
SPSA & Simultaneous-perturbation stochastic approximation & \S5.1 \\
VI & Value iteration & \S5.1 \\
\bottomrule
\end{tabular}
\end{table}

\bibliographystyle{informs2014}
\bibliography{references}

\end{document}